\documentclass[11pt]{article}
\usepackage{amssymb,amsfonts,amsmath,amsthm}\allowdisplaybreaks[4]
\usepackage{mathrsfs}
\usepackage{anysize}
\usepackage{hyperref}
\usepackage{a4, fullpage}
\usepackage{graphicx}
\usepackage{tikz}

\begin{document}
\baselineskip=14pt

\makeatletter\@addtoreset{equation}{section}
\makeatother\def\theequation{\thesection.\arabic{equation}}

\newcommand{\todo}[1]{{\tt [TODO: {#1}]}}

\newtheorem{defin}{Definition}[section]
\newtheorem{teo}{Theorem}[section]
\newtheorem{Prop}[teo]{Proposition}
\newtheorem{ml}{Main Lemma}
\newtheorem{con}{Conjecture}
\newtheorem{cond}{Condition}
\newtheorem{conj}{Conjecture}
\newtheorem{prop}[teo]{Proposition}
\newtheorem{lem}[teo]{Lemma}
\newtheorem{rmk}[teo]{Remark}
\newtheorem{cor}[teo]{Corollary}
\newtheorem{ass}[teo]{Assumption}

\newcommand{\be}{\begin{equation}}
\newcommand{\ee}{\end{equation}}
\newcommand{\ben}{\begin{eqnarray}}
\newcommand{\benn}{\begin{eqnarray*}}
\newcommand{\een}{\end{eqnarray}}
\newcommand{\eenn}{\end{eqnarray*}}
\newcommand{\bp}{\begin{prop}}
\newcommand{\ep}{\end{prop}}
\newcommand{\bt}{\begin{teo}}
\newcommand{\et}{\end{teo}}
\newcommand{\bcor}{\begin{cor}}
\newcommand{\ecor}{\end{cor}}
\newcommand{\bcon}{\begin{con}}
\newcommand{\econ}{\end{con}}
\newcommand{\bcond}{\begin{cond}}
\newcommand{\econd}{\end{cond}}
\newcommand{\br}{\begin{rmk}}
\newcommand{\er}{\end{rmk}}
\newcommand{\bl}{\begin{lem}}
\newcommand{\el}{\end{lem}}
\newcommand{\bit}{\begin{itemize}}
\newcommand{\eit}{\end{itemize}}
\newcommand{\bd}{\begin{defin}}
\newcommand{\ed}{\end{defin}}
\newcommand{\bpr}{\begin{proof}}
\newcommand{\epr}{\end{proof}}

\newcommand{\Z}{{\mathbb Z}}
\newcommand{\R}{{\mathbb R}}
\newcommand{\E}{{\mathbb E}}
\newcommand{\C}{{\mathbb C}}
\renewcommand{\P}{{\mathbb P}}
\newcommand{\N}{{\mathbb N}}

\newcommand{\Bi}{{\cal B}}
\newcommand{\Si}{{\cal S}}
\newcommand{\Ti}{{\cal T}}
\newcommand{\Wi}{{\cal W}}
\newcommand{\Yi}{{\cal Y}}
\newcommand{\Hi}{{\cal H}}
\newcommand{\Fi}{{\cal F}}
\newcommand{\Zi}{{\cal Z}}

\newcommand{\eps}{\epsilon}

\newcommand{\nn}{\nonumber}

\newcommand{\pa}{\partial}
\newcommand{\ffrac}[2]{{\textstyle\frac{{#1}}{{#2}}}}
\newcommand{\dif}[1]{\ffrac{\partial}{\partial{#1}}}
\newcommand{\diff}[1]{\ffrac{\partial^2}{{\partial{#1}}^2}}
\newcommand{\difif}[2]{\ffrac{\partial^2}{\partial{#1}\partial{#2}}}

\newcommand{\asto}[1]{\underset{{#1}\to\infty}{\longrightarrow}}
\newcommand{\Asto}[1]{\underset{{#1}\to\infty}{\Longrightarrow}}
\newcommand{\astoo}[1]{\underset{{#1}\to 0}{\longrightarrow}}
\newcommand{\Astoo}[1]{\underset{{#1}\to 0}{\Longrightarrow}}

\newcommand{\treeroot}{\varnothing}

\newcommand\NOTE[1]{{\color{blue}\sf [#1]}}
\newcommand\NOTEM[1]{\marginpar{\raggedright\color{blue}\sf\small #1}}

\title{Low-dimensional lonely branching random walks die out}
\author{Matthias Birkner$^{\,1}$, Rongfeng Sun$^{\,2}$}
\date{\today}
\maketitle

\footnotetext[1]{Institut f\"ur Mathematik, Johannes-Gutenberg-Universit\"at Mainz, Staudingerweg 9, 55099 Mainz, Germany. Email:
birkner@mathematik.uni-mainz.de}

\footnotetext[2]{Department of Mathematics, National University of Singapore, 10 Lower Kent Ridge Road, 119076 Singapore. Email:
matsr@nus.edu.sg}

\begin{abstract}
  The lonely branching random walks on $\Z^d$ is an interacting
  particle system where each particle moves as an independent random
  walk and undergoes critical binary branching when it is alone.  We
  show that 
  if the symmetrized walk is recurrent, lonely branching random walks
  die out locally. Furthermore, the same result holds if additional
  branching is allowed when the walk is not alone.
\end{abstract}

\section{Model and result}

We consider systems of (critical binary) \emph{lonely branching random
  walks}: Particles move as independent continuous-time irreducible
random walks on $\Z^d$ with jump rate $1$, jumps are taken according
to a probability kernel $p_{xy}=p_{y-x}$, $x,y\in\Z^d$. In addition,
whenever a particle is alone at its site, it undergoes critical binary
branching at rate $\gamma$. We will denote the particle configuration
at time $t$ by $\eta(t):=(\eta_x(t))_{x\in \Z^d}$, with $\eta_x(t)$
being the number of particles at site $x$ at time $t$. For
$\eta=(\eta_x)_{x\in \Z^d} \in \N_0^{\Z^d}$ and (suitable) test
functions $f : \N_0^{\Z^d}\to \R$, the generator is (formally) given by
\begin{align}
\label{generatorcatalyticbranchers}
 L\,f(\eta) & = \sum_{x, y} \eta_x p_{x y} 
 \big( f(\eta^{x \to y}) - f(\eta) \big) 
 + \gamma \sum_{x} 1_{\{\eta_x=1\}} \frac12 
 \big( f(\eta^{+x}\eta) +  f(\eta^{-x}\eta)- 2f(\eta) 
 \big) . 
\end{align}
where 
\begin{align}
  \label{eq:eta.moves}
  \eta^{x \to y} := \eta+\delta_y-\delta_x, \quad 
  \eta^{+x} := \eta+\delta_x, \quad \eta^{-x} := \eta-\delta_x
\end{align}
(from $\eta$, $\eta^{x \to y}$ arises by moving a particle from $x$ to $y$, 
$\eta^{+x}$ arises by adding a particle at site $x$ and $\eta^{-x}$ arises by removing a particle at $x$).

Using monotonicity and approximations with finite initial conditions, one can 
start the process $(\eta(t))_{t \ge 0}$ from any initial condition $\eta(0) \in \N_0^{\Z^d}$. It is 
then -- analogous to systems of independent random walks -- in principle possible that the system explodes in finite time 
in the sense that the number of particles at some site becomes infinite. However, we will only 
consider (possibly random) initial conditions for which the system is well-defined and locally finite for all times 
(this is amply guaranteed by Assumption~\eqref{eq:boundedintensity} in Thm.~\ref{thm1}). 
We discuss the rigorous construction of the process with pointers to the literature in 
Remark~\ref{rmk:rigorousconstruction} below. 
\begin{ass}
\label{ass:p}
The probability kernel $(p_x)_{x\in\Z^d}$ is irreducible and the
random walk with the symmetrised jump kernel
$\widehat{p}_x := (p_x+p_{-x})/2$ is recurrent.
\end{ass}
Note that if $p$ has finite second moments, Assumption~\ref{ass:p} is equivalent to $d\leq 2$.

\begin{teo} 
  \label{thm1}
  If $p=(p_x)_{x\in\Z^d}$ satisfies Assumption~\ref{ass:p},
  the branching rate $\gamma>0$, and 
  \begin{equation}
    \label{eq:boundedintensity}
    \sup_{x\in\Z^d} \E[\eta_x(0)]<\infty 
  \end{equation}
  holds, then the
  lonely branching random walks die out locally in probability, i.e.,
  \begin{align}
    \label{thm1:eq}
    \lim_{t\to\infty} \P(\eta_x(t)=0) = 1 \quad \text{for all }x \in \Z^d. 
  \end{align}
\end{teo}

\begin{rmk}\rm
  A simple coupling argument (see \cite[Lemma~1, Ch.~2.2]{B03}) shows that
  $\eta$ is a monotone process, thanks to the binary
  branching. Therefore it suffices to prove Theorem~\ref{thm1} under
  the assumption that $\E[\eta_x(0)]$ is constant in $x\in\Z^d$. We
  assume this from now on.
\end{rmk}

\begin{rmk}[Construction of the process and suitable state spaces]\rm
  \label{rmk:rigorousconstruction}
  The Markov process $(\eta(t))_{t\ge 0}$ can be obtained in a fairly straightforward way 
  as a solution to an infinite system of Poisson-process driven stochastic equations, 
  see \cite[Ch.~2.2]{B03} for a rigorous construction. 
  \smallskip
  
  $(\eta(t))_{t \ge 0}$ is locally finite and well defined for any initial
  configuration $\eta(0)$ from
  \begin{align} 
    \label{eq:condIC0}
    \mathbb{S}_\mathrm{max} := \Big\{ \eta \in \N_0^{\Z^d} \colon
    \sum\nolimits_{y \in \Z^d} \eta_y p_{yx}(t) < \infty \; \text{for all } t \ge 0, x \in \Z^d \Big\} 
  \end{align}
  where $(p_{xy}(t))_{x,y\in\Z^d}$ is the transition kernel of the
  random walk at time $t$: Comparison with supercritical binary
  branching random walks (particles split in two at rate $\gamma$)
  shows that then $\E\big[\eta_x(t) | \eta(0)\big] \le e^{\gamma t}
  \sum_y \eta_y(0) p_{yx}(t) < \infty$ for any $x \in \Z^d$, $t \ge
  0$, in particular, there is no
  explosion. Assumption~\eqref{eq:boundedintensity} implies $\eta(0)
  \in \mathbb{S}_\mathrm{max}$ a.s.

  If \eqref{eq:condIC0} is violated for a certain initial condition $\eta(0)$, 
  i.e.\ $\sum_y \eta_y(0) p_{y
    x_0}(t_0)=\infty$ for some $t_0$ and $x_0$, then by
  irreducibility, the system will explode everywhere by time $t_0+$.
  (Note that the number of particles at $x_0$ at time $t'$,
  which did not undergo any branching in the time interval $[0,t']$, is 
  bounded from below by the sum of independent indicators with total mean 
  $e^{-\gamma t'} \sum_y \eta_y(0) p_{y x_0}(t')$.)
  \medskip
  
  For computations involving the generator \eqref{generatorcatalyticbranchers}, it is 
  more convenient to restrict to a smaller set of allowed initial conditions, 
  which is still large enough for the purposes of this note: 
  Pick some 
  reference weight sequence $(w_x)_{x \in \Z^d} \subset 
  (0,\infty)^{\Z^d}$ with the property 
  \begin{align} 
    \label{eq:LSweightproperties}
    \sum\nolimits_{x\in\Z^d} w_x < \infty \quad \text{and} \quad 
    \sum\nolimits_{y \in \Z^d} p_{x y} w_y \le M w_x, \quad x \in \Z^d
  \end{align}
  for some $M < \infty$, which implies $\sum_{y \in \Z^d} p_{x y}(t)
  w_y \le e^{M t} w_x$ for $t \ge 0$, $x \in \Z^d$. 
  [A simple choice, following \cite{LS81}, is $w_x = \sum_{n=0}^\infty
  M^{-n} \sum_y p^{(n)}_{x y} v_y$ for some $M>1$ and a summable and
  strictly positive sequence $(v_x)_{x \in \Z^d} \subset
  (0,\infty)^{\Z^d}$, where $p^{(n)}_{x y}$ denotes entry $(x,y)$ of
  the $n$-th power of $p$.]

  Let 
  \begin{align} 
    \label{def:wnorm}
    \mathbb{S}_w := \big\{ \eta \in \N_0^{\Z^d} \colon || \eta ||_w < \infty \big\}, 
    \quad \text{where }
    || \eta ||_w := \sum\nolimits_{x \in \Z^d}\eta_x w_x .
  \end{align}
  $\mathbb{S}_w$ is (a closed subset of) a weighted $\ell_1$-space, equipped with $||\cdot||_w$ it is a 
  complete and separable metric space; 
  $\mathbb{S}_w \subset \mathbb{S}_\mathrm{max}$ for any such choice of $(w_x)$. 

  Write $\mathrm{Lip}(\mathbb{S}_w)$ for the 
  Lipschitz continuous functions on $\mathbb{S}_w$. It follows from the computations 
  in \cite[Section~2.2]{B03} that for 
  $f \in \mathrm{Lip}(\mathbb{S}_w)$ there exists $c_f < \infty$ such that 
  \begin{align} 
    \label{eq:Lnormbound}
    \big| L f(\eta) \big| \le c_f ||\eta||_w \quad 
    \text{for all } \eta \in \mathbb{S}_w
  \end{align}
  and that $\mathrm{Lip}(\mathbb{S}_w)$ is a core for $L$ from
  \eqref{generatorcatalyticbranchers}. In particular
  (see, e.g., \cite[Lemma~3]{B03}) there is a constant $C=C(w) < \infty$ such
  that
  \begin{equation} 
    \E\big[ ||\eta(t)||_w \big] \le e^{C t} \E\big[ ||\eta(0)||_w \big] 
    \quad \text{for all } t \ge 0
  \end{equation}
  and 
  \begin{equation}
    \label{eq:Epsieta0finite}
    \E\big[ ||\eta(0)||_w \big] < \infty 
  \end{equation}
  implies that $\eta(t) \in \mathbb{S}_w$ for all $t$. 
  Note that \eqref{eq:boundedintensity} implies \eqref{eq:Epsieta0finite}.
\end{rmk}

\subsection{Discussion} 
\label{sect:discussion}

The system \eqref{generatorcatalyticbranchers} is a special case of self-catalytic 
critical binary branching random walks (SCBRW${}_b$) on $\Z^d$ where each 
particle independently performs a random walk with kernel $p$ 
and in addition while there are $k-1$ other particles at its site, 
it splits in two or disappears with rate $b(k)$, where $b : \N_0 \to 
[0,\infty)$ is the branching rate function, i.e., the second 
sum on the right-hand side of 
\eqref{generatorcatalyticbranchers} is replaced by 
$L_{br}^{(b)} f(\eta) = \sum_{x} b(\eta_x) \frac12 \big( f(\eta^{+x}) + f(\eta^{-x})- 2f(\eta) \big)$.
The choice $b=\gamma \mathbf{1}_{\{k=1\}}$ leads to \eqref{generatorcatalyticbranchers}.

By the comparison result from \cite[Thm.~1 and Cor.~1, Ch.~2.7]{B03} 
(a discrete particle analogue of the main result from \cite{CFG96}), 
Theorem~\ref{thm1} implies 
\begin{cor}If $p$ satisfies Assumption~\ref{ass:p} and $\sup_{x\in\Z^d} \E[\eta^{(b)}_x(0)]<\infty$, then
the SCBRW${}_b$ $(\eta^{(b)}(t))_{t\geq 0}$ with branching rate function $b$ will die out locally whenever $b(1)>0$.
\end{cor}
This confirms \cite[Conjecture~1, Ch.~2.8]{B03}, which was also
formulated by Alison Etheridge (personal communication),
for 
recurrent random walks satisfying Assumption~\ref{ass:p}.
\medskip

The case $b(k) = ck$ for some $c>0$ corresponds 
to classical systems of independent branching random walks (IBRW). For IBRW, 
local extinction in ``low dimensions'', i.e.\ when the underlying symmetrised random walk is 
recurrent, is well known, 
\cite{K77}, 
\cite{D77}, 
\cite{F75}. 
In fact, the low-dimensional 
IBRW exhibit ``clustering'' -- local extinction combined with increasingly rare regions 
of diverging particle density. 
See also \cite{L-MW94} for references and discussion concerning persistence vs.\ local extinction for 
independent branching random walks in various contexts. 

These papers do make use of the independence properties inherent in IBRW (different families 
evolve independently), which is not the case in our system(s). 
In particular, our arguments do not (and can not) rely on explicit computations 
or estimates for Laplace transforms. 

Our proof technique for Theorem~\ref{thm1} is in so far inspired by
\cite{K77} that we show clustering by analysing a suitable stochastic
representation of the Palm distribution. In the context of IBRW and its
relatives, related ``Kallenberg tree'' constructions for critical
spatial systems have been used e.g.\ in \cite{GRW90},
\cite{GW91}, \cite{GRW92}, \cite{GW94} and similarly, ``spine''
constructions for supercritical branching processes have been
considered in the literature, e.g.\ \cite{EK04} and references there
(see also \cite{LPP95} and discussion of references there on p.~1129). 
Arguably, the present manuscript highlights the robustness and usefulness 
of this type of stochastic representation, especially when more analytic 
tools are unavailable because of inter-dependence of different families.
\medskip

Even under Assumption~\ref{ass:p}, one can set up initial conditions
$\eta(0)$ such that \eqref{thm1:eq} and \eqref{eq:boundedintensity}
both fail. For example, take for $p$ symmetric simple random walk on
$\Z^1$ and make $\eta_x(0) \approx e^{c |x|}$ grow to $\infty$ as
$|x|\to\infty$ so that $\eta(0) \in \mathbb{S}_\mathrm{max}$ but the
number of particles which reach $0$ at time $t$ without having
branched before does not converge to $0$ in probability. Obviously,
such initial conditions are not stationary in space and it seems
highly doubtful whether $\eta(t)$ would then converge to an equilibrium
concentrated on $\mathbb{S}_\mathrm{max}$.
Still, while Theorem~\ref{thm1} shows in particular that under Assumption~\ref{ass:p} 
there can be no non-trivial equilibria with finite intensity, it does 
not rule out the possibility of equilibria with infinite intensity.
It is known that this is not the case for IBRW, see \cite{BCG93} 
(there, literally proved for branching Brownian motion and super Brownian motion, 
using comparison arguments for the Laplace transforms). For SCBRW, this question 
remains open at the moment.

When $\widehat{p}$ is transient, there is a family of non-trivial
equilibria, parametrised by the average particle density, analogous to
the case of IBRW, see \cite[Prop.~3]{B03}.
 \medskip

In \cite{BS14}, we considered the following caricature 
of the system $\xi$ from Section~\ref{subsect:relviewpoint},
originally proposed by Anton Wakolbinger: Replace the random walk
special path by a constant path and disallow branching away from the
special path but keep the immigration mechanism along it 
unchanged (``random walks with self-blocking immigration'').
The main results from \cite{BS14} corroborate Theorem~\ref{thm1} in a
quantitative way, and in fact lead to the conjecture that in $d=1$
and assuming that $p$ has finite second moments, 
the typical number of particles at the origin
under the Palm distribution of the lonely branching random walks should diverge like $\log t$ in $d=1$.  However,
undoing the caricature steps to convert our findings into an actual
proof of this conjecture will require new arguments.
\bigskip

\noindent
In Section~\ref{sect:Palm}, we introduce the stochastic representation
of the locally size-biased (or ``Palm'') law of $\eta$; its behaviour
is analysed in Section~\ref{sect:proofthm1}, which completes the proof
of Theorem~\ref{thm1}. 


\section{The locally size-biased process}
\label{sect:Palm}

The key to proving Theorem \ref{thm1} is to study the locally size-biased law of $\eta$, which we introduce below. 

\subsection{The locally size-biased process $\widehat \eta^{(x,T)}$ 
as a main ingredient for the proof of Theorem~\ref{thm1}}

For $x\in\Z^d$ and $T \ge 0$, assume that 
\begin{equation}
  \label{eq:EetaxTfinite}
  \sum_y \E[\eta_y(0)] p_{yx}(T) < \infty.
\end{equation} 
\eqref{eq:EetaxTfinite} follows in particular from the assumption
\eqref{eq:boundedintensity} in Theorem~\ref{thm1} but this is the
``correct'' (and somewhat milder) assumption for the following
construction since the term in \eqref{eq:EetaxTfinite} equals
$\E[\eta_x(T)]$. 
\smallskip

Let $\widehat{\eta}^{(x,T)}:=(\widehat{\eta}^{(x,T)}(t))_{0\leq t\leq T}$ have the locally size-biased (w.r.t.\ $\eta_x(T)$) distribution of $\eta:=(\eta(t))_{0\leq t\leq T}$, i.e.\ 
\begin{equation}\label{lsbias}
\E[f(\widehat{\eta}^{(x,T)})] = \frac{\E[\eta_x(T) f(\eta)]}{\E[\eta_x(T)]}
\end{equation}
for any (say, bounded or non-negative) test function $f$.  
We will show that for every $x\in\Z^d$, 
\begin{align}
\label{eq:sbdiv}
\inf_{K \ge 0} \liminf_{T\to\infty} \P(\widehat{\eta}^{(x,T)}_x(T) \ge K) = 1,
\end{align}
i.e.\ $\widehat{\eta}^{(x,T)}_x(T) \to \infty$ in distribution. Since under 
the assumptions of Theorem~\ref{thm1} (see \cite[Lemma~4a)]{B03}), 
\[
\sup_{t\ge 0} \E[\eta_x(t)] = \sup_{t\ge 0} \sum_y \E[\eta_y(0)] p_{yx}(t) 
\le \sup_y \E[\eta_y(0)] < \infty,
\]
\eqref{eq:sbdiv} implies Theorem~\ref{thm1} by a standard argument. Indeed, by \eqref{lsbias} with $f(\eta)=\frac{1_{\{\eta_x(T)\geq 1\}}}{\eta_x(T)}$,
\begin{align*}
\P(\eta_x(T) \ge 1) 
= \E[\eta_x(T)] \times \E\Bigg[ \frac{1_{\{\widehat{\eta}^{(x,T)}_x(T) \ge 1\}}}{\widehat{\eta}^{(x,T)}_x(T)}  \Bigg] \le \big( \sup_{t\ge 0} \E[\eta_x(t)] \big) \times 
\Big( \P\big(1 \le \widehat{\eta}^{(x,T)}_x (T) \le K\big) + \frac1K \Big)
\end{align*}
for every $K>1$. Taking $T\to\infty$ followed by $K \to \infty$ then implies $\lim_{T\to\infty} \P(\eta_x(T) \ge 1) = 0$.

\subsection{A stochastic representation of $\widehat{\eta}^{(x, T)}$} 
\label{sect:stochrephateta}

\setcounter{figure}{0}
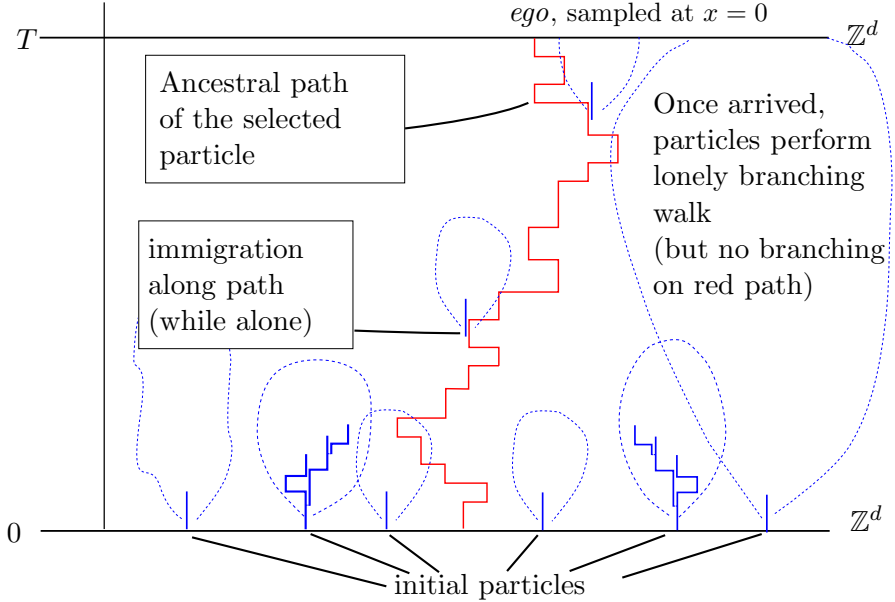
\begin{figure}[tp]
\begin{center}
\definecolor{cff0000}{RGB}{255,0,0}
\definecolor{c0000ff}{RGB}{0,0,255}
\definecolor{cffffff}{RGB}{255,255,255}
\begin{tikzpicture}[y=1.15pt,x=1.15pt,yscale=-1, inner sep=0pt, outer sep=0pt]
\begin{scope}[shift={(-13.52575,-275.93147)}]
  \path[draw=black,line join=miter,line cap=butt,miter limit=4.00,line
    width=0.681pt] (23.6469,450.4656) .. controls (85.7036,450.4656) and
    (280.3914,450.4656) .. (280.3914,450.4656);

  \path[draw=black,line join=miter,line cap=butt,miter limit=4.00,line
    width=0.681pt] (23.3275,288.6313) .. controls (85.3842,288.6313) and
    (280.0720,288.6313) .. (280.0720,288.6313);

  \path[draw=black,line join=miter,line cap=butt,miter limit=4.00,line
    width=0.477pt] (44.3324,449.8572) -- (44.3324,277.0718);

  \path[fill=black] (285.86694,291.06494) node[above right] (text4413)
    {$\Z^d$};

  \path[fill=black] (287.08374,450.46558) node[above right] (text4417)
    {$\Z^d$};

  \path[fill=black] (12.695672,454.11597) node[above right] (text4421) {$0$};

  \path[fill=black] (15.737669,292.28177) node[above right] (text4425) {$T$};

  \path[fill=black] (218,285) node[above] {\emph{ego}{\small, sampled at $x=0$}};

  \path[shift={(13.6356,275.93147)},draw=black,line join=miter,line cap=butt,miter
    limit=4.00,line width=0.800pt] (128.5970,42.4884) .. controls
    (159.9910,38.6496) and (168.7103,34.0680) .. (168.7103,34.0680);

  \path[draw=black,miter limit=4.00,line width=0.352pt,rounded corners=0.0000cm]
    (57.8045,294.4127) rectangle (142.1082,336.2943);

  \path[fill=black] (62.287006,332) node[above right] (text4671) 
  { \parbox[b]{8em}{Ancestral path\\ of the selected\\ particle} }; 

  \path[draw=cff0000,line join=miter,line cap=butt,miter limit=4.00,line
    width=0.541pt] (172.7946,396.2723) -- (162.9099,396.2723) --
    (162.9099,403.9045) -- (155.4658,403.7853) -- (155.4658,413.4448) --
    (139.7234,413.4448) -- (139.7234,419.6460) -- (147.4115,419.6460) --
    (147.4115,428.5901) -- (155.2217,428.5901) -- (155.2217,434.7912) --
    (168.8895,434.7912) -- (168.8895,440.9924) -- (161.2014,440.9924) --
    (161.2014,449.6980)(184.3657,288.7980) -- (184.3657,294.8557) --
    (194.0515,294.8557) -- (194.0515,303.9423) -- (184.3657,303.9423) --
    (184.3657,310.0000) -- (201.8002,310.0000) -- (201.8002,320.6010) --
    (211.4861,320.6010) -- (211.4861,329.6876) -- (201.8002,329.6876) --
    (201.8002,335.7453) -- (192.1144,335.7453) -- (192.1144,350.8896) --
    (182.4285,350.8896) -- (182.4285,361.4906) -- (192.1144,361.4906) --
    (192.1144,372.0916) -- (172.7426,372.0916) -- (172.7426,381.1782) --
    (163.0568,381.1782) -- (163.0568,390.2647) -- (172.7426,390.2647) --
    (172.7426,396.3225);

  \begin{scope}[cm={{0.80473,0.0,0.0,0.77114,(19.15523,102.86709)}}]
    \path[shift={(13.6356,275.93147)},draw=c0000ff,line join=miter,line
      cap=butt,line width=0.800pt] (99.1071,174.1400) -- (99.1071,158.0686) --
      (91.0714,158.0686) -- (91.0714,151.4614) -- (99.2857,151.4614) --
      (99.2857,142.1757);

    \path[shift={(13.6356,275.93147)},draw=c0000ff,line join=miter,line
      cap=butt,line width=0.800pt] (100.7143,164.1400) -- (100.7143,148.7829) --
      (107.8571,148.7829) -- (107.8571,134.3186);

    \path[shift={(13.6356,275.93147)},draw=c0000ff,line join=miter,line
      cap=butt,line width=0.800pt] (109.4643,142.1757) -- (109.4643,137.7114) --
      (116.2500,137.7114) -- (116.2500,129.4972);

    \path[shift={(13.6356,275.93147)},draw=c0000ff,dash pattern=on 0.64pt off
      0.32pt,line join=miter,line cap=butt,miter limit=4.00,line width=0.320pt]
      (99.1964,163.8722) -- (100.7143,163.8722);

    \path[shift={(13.6356,275.93147)},draw=c0000ff,dash pattern=on 0.64pt off
      0.32pt,line join=miter,line cap=butt,miter limit=4.00,line width=0.320pt]
      (108.0357,141.9972) -- (109.2857,141.9972);

    \begin{scope}
      \path[draw=c0000ff,dash pattern=on 1.16pt off 1.16pt,line join=miter,line
        cap=butt,miter limit=4.00,line width=0.387pt] (112.1098,444.5165) .. controls
        (112.1098,444.5165) and (101.9875,434.8356) .. (98.8319,428.7557) .. controls
        (95.3992,422.1419) and (92.9632,415.0413) .. (93.1765,407.1755) .. controls
        (93.3898,399.3096) and (94.7489,389.1943) .. (100.7990,383.4130) .. controls
        (105.5693,378.8546) and (113.3528,377.6096) .. (119.9783,378.0786) .. controls
        (125.6144,378.4775) and (131.9597,380.5006) .. (135.4692,384.8678) .. controls
        (139.5625,389.9616) and (139.3398,397.5266) .. (139.1575,404.0233) .. controls
        (138.8842,413.7622) and (137.5618,424.3227) .. (132.0268,432.3928) .. controls
        (128.1780,438.0042) and (115.3064,444.2740) .. (115.3064,444.2740);

    \end{scope}
  \end{scope}
  \path[draw=c0000ff,line join=miter,line cap=butt,line width=0.630pt]
    (230.8230,450.2295) -- (230.8230,437.8362) -- (237.2895,437.8362) --
    (237.2895,432.7412) -- (230.6793,432.7412) -- (230.6793,425.5807);

  \path[draw=c0000ff,line join=miter,line cap=butt,line width=0.630pt]
    (229.5297,442.5181) -- (229.5297,430.6757) -- (223.7816,430.6757) --
    (223.7816,419.5217);

  \path[draw=c0000ff,line join=miter,line cap=butt,line width=0.630pt]
    (222.4883,425.5807) -- (222.4883,422.1381) -- (217.0276,422.1381) --
    (217.0276,415.8038);

  \path[draw=c0000ff,dash pattern=on 0.50pt off 0.25pt,line join=miter,line
    cap=butt,miter limit=4.00,line width=0.252pt] (230.7511,442.3116) --
    (229.5296,442.3116);

  \path[draw=c0000ff,dash pattern=on 0.50pt off 0.25pt,line join=miter,line
    cap=butt,miter limit=4.00,line width=0.252pt] (223.6379,425.4430) --
    (222.6320,425.4430);

  \begin{scope}[cm={{-0.80473,0.0,0.0,0.77114,(321.55014,103.1632)}}]
    \path[draw=c0000ff,dash pattern=on 1.16pt off 1.16pt,line join=miter,line
      cap=butt,miter limit=4.00,line width=0.387pt] (112.1098,444.5165) .. controls
      (112.1098,444.5165) and (101.9875,434.8356) .. (98.8319,428.7557) .. controls
      (95.3992,422.1419) and (92.9632,415.0413) .. (93.1765,407.1755) .. controls
      (93.3898,399.3096) and (97.0634,382.2877) .. (102.8388,376.2082) .. controls
      (107.1927,371.6251) and (109.4988,365.9849) .. (122.0181,370.8738) .. controls
      (126.2889,372.5416) and (130.1103,378.9909) .. (132.3310,384.2129) .. controls
      (135.5310,391.7375) and (136.8872,400.4257) .. (136.0193,408.6081) .. controls
      (135.0275,417.9585) and (131.3365,427.2181) .. (126.0642,434.8489) .. controls
      (123.3259,438.8123) and (115.3064,444.2740) .. (115.3064,444.2740);

  \end{scope}
  \path[draw=c0000ff,line join=miter,line cap=butt,line width=0.630pt]
    (186.9984,450.2295) -- (186.9984,437.8362);

  \path[shift={(13.6356,275.93147)},draw=c0000ff,dash pattern=on 0.91pt off
    0.91pt,line join=miter,line cap=butt,miter limit=4.00,line width=0.305pt]
    (171.7259,170.8496) .. controls (171.7259,170.8496) and (166.0431,164.2590) ..
    (164.6549,160.2430) .. controls (163.4136,156.6520) and (163.3549,152.6554) ..
    (163.7710,148.8788) .. controls (164.2433,144.5910) and (164.2136,139.2277) ..
    (167.5591,136.5044) .. controls (170.3784,134.2094) and (177.9131,135.4942) ..
    (178.4182,135.4942) .. controls (178.9233,135.4942) and (185.7819,139.7701) ..
    (187.2570,143.4492) .. controls (188.9221,147.6020) and (187.7206,152.6092) ..
    (186.2469,156.8337) .. controls (184.3204,162.3560) and (176.3979,171.3546) ..
    (176.3979,171.3546);

  \path[draw=c0000ff,line join=miter,line cap=butt,line width=0.630pt]
    (136.1973,449.9081) -- (136.1973,437.5148);

  \path[draw=c0000ff,dash pattern=on 0.91pt off 0.91pt,line join=miter,line
    cap=butt,miter limit=4.00,line width=0.305pt] (134.5604,446.4597) .. controls
    (134.5604,446.4597) and (128.8775,439.8691) .. (127.4893,435.8530) .. controls
    (126.2480,432.2620) and (126.1894,428.2655) .. (126.6054,424.4888) .. controls
    (127.0778,420.2010) and (127.0481,414.8378) .. (130.3935,412.1145) .. controls
    (133.2128,409.8194) and (140.7476,411.1043) .. (141.2527,411.1043) .. controls
    (141.7577,411.1043) and (148.6164,415.3802) .. (150.0915,419.0593) .. controls
    (151.7566,423.2121) and (150.5551,428.2193) .. (149.0813,432.4438) .. controls
    (147.1548,437.9660) and (139.2324,446.9647) .. (139.2324,446.9647);

  \path[draw=c0000ff,line join=miter,line cap=butt,line width=0.630pt]
    (71.1973,449.9081) -- (71.1973,437.5148);

  \path[draw=c0000ff,dash pattern=on 0.91pt off 0.91pt,line join=miter,line
    cap=butt,miter limit=4.00,line width=0.305pt] (69.5604,446.4596) .. controls
    (69.5604,446.4596) and (65.0775,444.0846) .. (62.4893,435.8530) .. controls
    (59.1610,425.2673) and (51.9679,441.3215) .. (56.2483,411.2745) .. controls
    (58.8286,393.1620) and (53.8653,398.5699) .. (53.9649,394.2573) .. controls
    (55.0305,348.1187) and (83.9663,368.0015) .. (85.5384,394.6757) .. controls
    (85.5681,395.1799) and (83.6164,415.3801) .. (85.0915,419.0592) .. controls
    (86.7566,423.2121) and (85.5551,428.2193) .. (84.0813,432.4438) .. controls
    (82.1548,437.9660) and (74.2324,446.9647) .. (74.2324,446.9647);

  \path[draw=c0000ff,line join=miter,line cap=butt,line width=0.630pt]
    (259.9473,450.9795) -- (259.9473,438.5863);

  \path[shift={(13.6356,275.93147)},draw=c0000ff,dash pattern=on 0.91pt off
    0.91pt,line join=miter,line cap=butt,miter limit=4.00,line width=0.305pt]
    (244.7095,170.3445) -- (227.5369,145.8483) .. controls (227.5369,145.8483) and
    (211.8603,115.8249) .. (206.3819,98.7895) .. controls (202.3083,86.1229) and
    (193.0544,62.7194) .. (196.0959,44.5805) .. controls (197.3377,37.1744) and
    (202.8533,31.0895) .. (207.2855,25.0275) .. controls (209.2969,22.2765) and
    (211.5758,19.6705) .. (214.2010,17.4973) .. controls (216.3783,15.6948) and
    (219.5019,15.2210) .. (221.4078,13.0283)(265.9227,12.7607) .. controls
    (278.7565,17.1269) and (268.8718,13.2810) .. (277.5394,19.5792) .. controls
    (285.4275,25.3110) and (287.7105,54.5896) .. (286.8833,64.7836) .. controls
    (285.8606,77.3870) and (286.4814,91.1036) .. (285.1156,103.6744) .. controls
    (283.5118,118.4349) and (281.0240,133.9587) .. (273.2463,146.6059) .. controls
    (267.3463,156.1997) and (248.2450,169.3344) .. (248.2450,169.3344);

  \path[draw=c0000ff,line join=miter,line cap=butt,line width=0.630pt]
    (161.9626,386.7110) -- (161.9626,374.3177);

  \path[draw=c0000ff,dash pattern=on 0.91pt off 0.91pt,line join=miter,line
    cap=butt,miter limit=4.00,line width=0.305pt] (160.3257,383.2626) .. controls
    (160.3257,383.2626) and (154.6428,376.6720) .. (153.2546,372.6560) .. controls
    (152.0133,369.0649) and (151.9547,365.0684) .. (152.3707,361.2917) .. controls
    (152.8431,357.0040) and (152.8134,351.6407) .. (156.1588,348.9174) .. controls
    (158.9781,346.6223) and (166.5129,347.9072) .. (167.0179,347.9072) .. controls
    (167.5230,347.9072) and (174.3816,352.1831) .. (175.8568,355.8622) .. controls
    (177.5218,360.0150) and (176.3204,365.0222) .. (174.8466,369.2467) .. controls
    (172.9201,374.7690) and (164.9976,383.7676) .. (164.9976,383.7676);

  \path[draw=c0000ff,line join=miter,line cap=butt,line width=0.630pt]
    (202.9830,315.6224) -- (202.9830,303.2291);

  \path[draw=c0000ff,dash pattern=on 0.91pt off 0.91pt,line join=miter,line
    cap=butt,miter limit=4.00,line width=0.305pt] (201.3461,312.1739) .. controls
    (201.3461,312.1739) and (195.6633,305.5833) .. (194.2750,301.5673) .. controls
    (193.0337,297.9763) and (192.0281,292.4645) ..
    (192.4441,288.6879)(218.3293,288.4353) .. controls (218.7317,292.7144) and
    (217.2145,296.8377) .. (215.7408,301.0623) .. controls (213.8143,306.5845) and
    (206.0181,312.6790) .. (206.0181,312.6790);

  \path[draw=black,fill=cffffff,miter limit=4.00,line width=0.320pt,rounded
    corners=0.0000cm] (55.5711,348.7537) rectangle (125.2716,390.6750);

  \path[draw=black,line join=miter,line cap=butt,miter limit=4.00,line
    width=0.800pt] (125.7613,384.7464) .. controls (152.2599,385.7380) and
    (160.1603,385.6117) .. (160.1603,385.6117);

  \path[fill=black] (58.635601,387) node[above right] (text4097)
  { \parbox{8em}{immigration\\along path\\(while alone)} } ;

  \path[fill=black] (138.6356,472.92862) node[above right] (text4109) {initial
    particles};

  \path[shift={(13.6356,275.93147)},draw=black,line join=miter,line cap=butt,line
    width=0.800pt] (122.5000,192.7114) -- (59.2857,176.6400);

  \path[shift={(13.6356,275.93147)},draw=black,line join=miter,line cap=butt,line
    width=0.800pt] (130.0000,188.4257) -- (97.8571,175.5686);

  \path[shift={(13.6356,275.93147)},draw=black,line join=miter,line cap=butt,line
    width=0.800pt] (139.6429,188.0686) -- (123.5714,175.9257);

  \path[shift={(13.6356,275.93147)},draw=black,line join=miter,line cap=butt,line
    width=0.800pt] (161.7857,187.7114) -- (171.7857,175.9257);

  \path[shift={(13.6356,275.93147)},draw=black,line join=miter,line cap=butt,line
    width=0.800pt] (182.1429,187.7114) -- (213.2143,175.9257);

  \path[shift={(13.6356,275.93147)},draw=black,line join=miter,line cap=butt,line
    width=0.800pt] (199.2857,189.8543) -- (244.6429,175.5686);

  \path[fill=black] (223,375) node[above right] 
  { \parbox{10em}{Once arrived,\\ particles perform\\lonely branching\\
  walk \\ (but no branching\\ on red path)} } ; 
\end{scope}
\end{tikzpicture}
\end{center}
\caption{Representation of the locally size-biased system $\widehat \eta^{(0,T)}$.} \label{figure}
\end{figure}

Given the locally size-biased process $\widehat{\eta}^{(x, T)}$, we
can select uniformly at random one of the particles at $x$ at time $T$
-- note that $\widehat{\eta}^{(x, T)}_x(T) \ge 1$ a.s.\ -- and denote
its ancestral path by $X:=(X_t)_{0\leq t\leq T}$. The pair
$(\widehat{\eta}^{(x, T)}, X)$ admits the following alternative
representation (see Figure~\ref{figure}), which will be the starting
point of our analysis.
\medskip

Pick $X(0)$ with distribution 
\begin{equation}\label{X}
\P(X(0)=y) = \frac{\E[\eta_y(0)] p_{yx}(T)}{\E[\eta_x(T)]}, \qquad y\in \Z^d. 
\end{equation}
Given $X(0)=y$, let $(X(t))_{0 \le t \le T}$ be a random walk (with kernel $p$) conditioned
to be at $x$ at time $T$, and let $\widetilde{\xi}^{(x,T)}(0)$ have the law of 
$\widehat{\eta}^{(y,0)}(0)$. Given the path $(X(t))_{0\leq t\leq T}$, the system 
$(\widetilde{\xi}^{(x,T)}(t))_{0 \le t \le T}$ evolves according to 
the dynamics of the lonely branching random walks, except that one of the 
particles at $X(0)$ at time $0$ becomes the ``selected particle'' and follows the path $X$. 
Whenever a branching event occurs for the selected particle, which happens with rate $\gamma$ 
while the selected particle is alone, it produces an offspring (i.e., it never dies). 
\begin{prop}
  \label{prop:etahatrepr}
  The pair $(\widetilde{\xi}^{(x,T)}, X)$ has the same distribution as
  $(\widehat{\eta}^{(x, T)}, X)$. In particular,
  \begin{align} 
    \label{eq:tildexi=hateta}
    \E[f(\widetilde{\xi}^{(x,T)})] = \frac{\E[\eta_x(T) f(\eta)]}{\E[\eta_x(T)]}
  \end{align}
holds for any non-negative measurable test function $f : \N_0^{\Z^d} \to [0,\infty)$.
\end{prop}
Proposition~\ref{prop:etahatrepr} is \cite[Prop.~5]{B03}, a proof via
a time-discretisation approximation was sketched there (the analogous
result in the discrete-time case can be achieved by a straightforward
calculation, see \cite[Lemma~8]{B03}).

Let us explain heuristically why such a representation holds. The discussion in \cite{B03} is more detailed;
we also present in Section~\ref{subsect:DoobTransform} below an alternative proof of Proposition~\ref{prop:etahatrepr} by interpreting the local size-biasing of $\eta(T)$ as a Doob transformation. 

For simplicity, assume that $\sum_y \eta_y(0)<\infty$ (the general case requires an additional approximation argument). Note that the particle configurations $(\eta(t))_{0\leq t\leq T}$ can be obtained from the family trees of all the ancestral particles at time 0, where the family tree ${\cal T}$ of an ancestral particle records the times of branching/death and the jumps of all its descendants. Let $\widehat {\cal T}:=\{\widehat{\cal T}_{y,i}\}_{y\in \Z^d, 1\leq i\leq \widehat\eta^{(x,T)}_y(0)}$ be the set of family trees generated by the size-biased lonely branching random walks $\widehat \eta^{(x,T)}$, and let $\widetilde {\cal T}:= \{\widetilde{\cal T}_{y,i}\}_{y\in \Z^d, 1\leq i\leq \widetilde\xi^{(x,T)}_y(0)}$ be the set of family trees generated by the
$\widetilde \xi^{(x,T)}$ process. To show that $(\widetilde{\xi}^{(x,T)}, X)$ has the same distribution as $(\widehat{\eta}^{(x, T)}, X)$, it suffices to show that $(\widetilde{\cal T}, X)$ and $(\widehat{\cal T}, X)$ have the same distribution. We refrain from formally defining the family trees. For a formalisation of a space of marked trees that could be used here see e.g.\ \cite{HR17} and the references there.

Given the family tree $S$ of an ancestral particle at time 0,  let $b(S)$, $d(S)$ and $j(S)$ denote respectively the set of times in $(0,T)$ when the ancestral particle or any of its descendants undergoes a branching, death, or a jump. For each $t\in j(S)$,
let $\Delta(t)\in \Z^d$ denote the associated jump increment. Let $l(S)$ denote the total time length of the family tree $S$ up to time $T$. For a selected path $X$ in the family tree $S$, let $b(X)$ and $j(X)$ denote the set of times in $(0,T)$ when $X$ undergoes a branching or a jump. 

Note that the probability density (w.r.t.\ product Lebesgue measure for the times of branching, death, and jumps) of $\widehat {\cal T}$ being equal to a given set of family trees $S=\{S_{z,i}\}_{z\in \Z^d, 1\leq i\leq \eta_z(0)}$, and $X$ following a given path $Y$ in $S_{y,1}$ with $Y(T)=x$, is equal to 
\begin{align*}
& f(S, Y) = 1_{\rm Adm}(S,Y) \frac{1 }{\eta_x(T)}
\frac{\P(\eta(0))\, \eta_x(T)}{\E[\eta_x(T)]} \prod_{z\in\Z^d} \eta_z(0)! \prod_{i=1}^{\eta_z(0)} \rho(S_{z,i}) \\
\mbox{with} \quad & \rho(S_{z,i}) = e^{- l(S_{z,i}) - \gamma l_{\mathrm{lon}}(S_{z,i},S)} 
\Big(\frac{\gamma}{2}\Big)^{b(S_{z,i})+d(S_{z,i})} 2^{b(S_{z,i})} \!\!\! \prod_{t\in j(S_{z,i})} \!\!\! p(\Delta(t)),
\end{align*}
where $(\eta(t))_{0\leq t\leq T}$ is the particle configuration generated by the family trees $S$, 
$l_{\mathrm{lon}}(S_{z,i},S)$ is the total ``lonely length'' of the family tree $S_{z,i}$ w.r.t.\ the 
whole set $S=\{S_{z',i'}\}_{z'\in \Z^d, 1\leq i'\leq \eta_z(0)}$ (i.e.\ the length of all those parts of the branches 
of the tree $S_{z,i}$ which correspond to a particle which is currently alone at its site), 
$1_{\rm Adm}$ ensures that $(S,Y)$ is an admissible configuration for the lonely branching random walks, the factor $1/\eta_x(T)$ accounts for the probability of selecting $Y$ among all $\eta_x(T)$ paths ending at $x$ at time $T$, the factor $\eta_z(0)!$ accounts for the symmetry in assigning the family trees $(S_{z,i})_{1\leq i\leq \eta_z(0)}$ to the $\eta_z(0)$ individuals at $z$ at time $0$, the exponential factor accounts for the absence of branching, death and jumps in $S_{z,i}$ except at the specified times,  the factor $\gamma/2$ is the probability density of a branching or death occurring at a specified time, and a factor $2$ is assigned to each branching to account for the symmetry in assigning sub-family trees to the two descendants. 

Similarly, we find that the probability density of $(\widetilde {\cal T}, X)$ being equal to $(S, Y)$ is given by 
\begin{align*}
g(S, Y) & = 1_{\rm Adm}(S,Y) \frac{ \E[\eta_y(0)] p_{yx}(T)}{\E[\eta_x(T)]}\cdot \frac{e^{-T -\gamma l_\mathrm{lon}(Y,S)}\gamma^{b(Y)}\prod_{t\in j(Y)} p(\Delta(t))}{p_{yx}(T)} \cdot \frac{\eta_y(0)\P(\eta(0))}{\E[\eta_y(0)]}  \\
& \quad \times  e^{-(l(S_{y,1})-T) - \gamma (l_\mathrm{lon}(S_{y,1},S)-l_\mathrm{lon}(Y,S))}
\Big(\frac{\gamma}{2}\Big)^{b(S_{y,1})-b(Y)+d(S_{y,1})} 2^{b(S_{y,1})-b(Y)} \!\!\! \prod_{t\in j(S_{y,1})\backslash j(Y)} \!\!\! p(\Delta(t)) \\
& \quad \times (\eta_y(0)-1)! \prod_{i=2}^{\eta_y(0)} \rho(S_{y,i}) \times \prod_{z\in\Z^d, z\neq y} \eta_z(0)! \prod_{i=1}^{\eta_z(0)} \rho(S_{z,i}).
\end{align*}
Observe that $f(S, Y)=g(S, Y)$, and hence $(\widetilde{\xi}^{(x,T)}, X)$ has the same distribution as $(\widehat{\eta}^{(x, T)}, X)$.

\subsubsection{Local size-biasing as a Doob-transform: Another proof of Proposition~\ref{prop:etahatrepr}}
\label{subsect:DoobTransform}

Proposition~\ref{prop:etahatrepr}
can be proved ``directly'' (and in a sense, ``purely algebraically''
using computations with the generator) without approximation
arguments, i.e., not using time-discretisation nor approximation by
finite systems.  This can be formulated in terms of a ``filtering
problem'' for an enriched Markov process that we briefly sketch here,
with more detailed computations relegated to
Appendix~\ref{sect:DoobTransform-aux}. 
\medskip

Fix $x_0 \in \Z^d$, $T>0$.
The function
\begin{equation}
  \label{eq:h.eta.t}
  h(\eta,t) := \sum_{z \in \Z^d} \eta_z p_{z,x_0}(T-t), \quad \eta \in \mathbb{S}_w, 0 \le t \le T
\end{equation}
solves $\big(L + (\partial/\partial t) \big) h (\eta,t) \equiv 0$ with
$L$ from \eqref{generatorcatalyticbranchers}, i.e.\ $h$ is space-time
harmonic for $(\eta_t)_{0 \le t \le T}$, see
\eqref{eq:hspacetimeharmonic} in
Appendix~\ref{sect:DoobTransform-aux}.  Thus, we can define the
$h$-transformed process $\big(\widehat{\eta}(t)\big)_{0 \le t \le T}$
with (time-inhomogeneous) generator
\begin{equation} 
  \label{eq:hatLt}
  \widehat{L}_t f (\eta, t) = \frac1{h(\eta,t)} \Big( \big(L + \tfrac{\partial}{\partial t} \big) h f \Big) (\eta,t)
\end{equation}
for $0 \le t < T$, $\eta \in \mathbb{S}_w$. With reference to Remark~\ref{rmk:rigorousconstruction}, 
we can use for example test functions 
$f : \mathbb{S}_w \times [0,T] \to \R$ such that $f(\cdot,t)$ and 
$\frac{\partial}{\partial t} f(\cdot,t)$ are both Lipschitz continuous uniformly in $t \in [0,T]$.
Note that by definition, for any (say, non-negative or bounded) test function $f$ 
\begin{align} 
  \E\big[ f\big((\widehat{\eta}(t))_{0 \leq t \leq T} \big)\big]
  & = \frac1{\E[h(\eta(0),0)]} \E\big[ h(\eta(T),T) f\big((\eta(t))_{0 \leq t \leq T}\big)\big]
    \notag \\
  & = \frac1{\E[\eta_{x_0}(T)]} \E\big[ \eta_{x_0}(T) f\big((\eta(t))_{0 \leq t \leq T}\big) \big], 
\end{align}
i.e., we have $\widehat{\eta} \mathop{=}^d \widehat\eta^{(x,T)}$ from \eqref{lsbias}.

Straightforward computation (see Appendix~\ref{sect:DoobTransform-aux}) yields 
\begin{align} 
\widehat{L}_t f (\eta, t) = &
\sum_{x, y\in\Z^d} \eta_x p_{xy}  \Big(1 - s_x(\eta,t) + s_x(\eta,t) \frac{p_{y,x_0}(T-t)}{p_{x,x_0}(T-t)} \Big)  
\big( f(\eta^{x \to y},t) - f(\eta,t) \big) \notag \\ 
& + \frac{\gamma}{2} \sum_{x} 1_{\{\eta_x=1\}} \Big( \big(1+s_x(\eta,t)\big) \big( f(\eta^{+x},t) - f(\eta,t) \big)
\notag \\[-2ex] 
& \hspace{8em} + \big(1-s_x(\eta,t)\big) \big( f(\eta^{-x},t) - f(\eta,t) \big) \Big) 
+ \frac{\partial}{\partial t} f(\eta,t)
\label{eq:Lhat-formula}
\end{align}
where 
\begin{align}
\label{eq:Lhat-formula.sx}
s_x(\eta,t) = \frac{p_{x,x_0}(T-t)}{h(\eta,t)} = 
\frac{p_{x,x_0}(T-t)}{\sum_z \eta_z p_{z,x_0}(T-t)} 
= \frac1{\eta_x} \frac{\eta_x p_{x,x_0}(T-t)}{\sum_z \eta_z p_{z,x_0}(T-t)} 
\end{align}
can be interpreted as 
the probability that, given $\widehat\eta(t)=\eta$, the selected particle is a particular 
particle at site $x$ at time $t$.

\paragraph{Enriched process including a selected path}
Note that the formulation of $\widehat{\eta}$ as a time-inhomogeneous Markov process with generator 
\eqref{eq:hatLt} does not literally contain a particle with a ``privileged status'', 
in contrast to our formulation at the beginning of Section~\ref{sect:stochrephateta}.

The statement in Proposition~\ref{prop:etahatrepr} includes the path
$X$ of the selected particle, and we can keep track of the ``tagged
position'' $X(t)$ where the selected particle currently sits in a
Markovian way. Indeed, the process
$(\widetilde{\xi}, X) = \big(\widetilde{\xi}(t), X(t)\big)_{0\le t \le T}$
from Proposition~\ref{prop:etahatrepr} is a time-inhomogeneous
Markov process with values in $\mathbb{S}_w \times \Z^d$ (more
precisely, only pairs $(\xi,z)$ with $\xi_z \ge 1$ are possible) and
generator
\begin{align}
  \widetilde{L}_t f(\xi,z,t) & = \sum_{x, y\in\Z^d} (\xi_x -\delta_{xz}) p_{xy} \big(f(\xi^{x\to y},z,t) - f(\xi,z,t)\big) 
  \notag \\
  & \hspace{1.5 em} 
  + \sum_{y\in\Z^d} p_{zy} \frac{p_{y,x_0}(T-t)}{p_{z,x_0}(T-t)} 
  \big(f(\xi^{z\to y},y,t) - f(\xi,z,t)\big) \notag \\
  & \hspace{1.5 em} + \frac{\gamma}{2} \sum_{x \neq z} 1_{\xi_x=1} \big( f(\xi^{+x},z,t) + f(\xi^{-x},z,t) - 2 f(\xi,z,t) \big) 
  \notag \\
  & \hspace{1.5 em}+ \gamma 1_{\xi_z=1} \big( f(\xi^{+z},z,t)  - f(\xi,z,t) \big) 
  + \frac{\partial}{\partial t} f(\xi,z,t)
\end{align}
Here, we can use test functions $f : \mathbb{S}_w \times \Z^d \times [0,T] \to \R$ such that $f(\cdot,z,t)$ and 
$\frac{\partial}{\partial t} f(\cdot,z,t)$ are both Lipschitz uniformly in $z \in \Z^d$ and 
$t \in [0,T]$. 
Strictly speaking, since some jump rates can become $\infty$ at $t=T-$ (namely, for $z\neq x_0$), 
we should restrict to subintervals $[0,T']$ with $T'<T$ first and then let finally $T' \nearrow T$; we will 
skip these details in the presentation.
%

\paragraph{Markov mapping}
Define the projection $\pi_{\mathbb{S}_w} : \mathbb{S}_w \times \Z^d \to \mathbb{S}_w$ 
with $\pi_{\mathbb{S}_w}\big( (\xi,z) \big) = \xi$.
Proposition~\ref{prop:etahatrepr} follows from the distributional identity 
\begin{align} 
  \label{eq:projtildeexi=hateta}
  \big( \pi_{\mathbb{S}_w}(\widetilde\xi(t), X(t)) \big)_{0 \le t \le T} 
  \mathop{=}^d \big(\widehat{\eta}(t)\big)_{0 \le t \le T} \, .
\end{align}

In fact, we have
\[
  \P\big( X(t) \in \cdot \,\big|\, \sigma(\widetilde\xi(s) : s \le t)\big)
  = \alpha_t(\widetilde\xi(t), \cdot),
\]
where for $0 \le t \le T$, the probability kernels $\alpha_t$ from $\mathbb{S}_w$ to 
$\mathcal{M}_1(\mathbb{S}_w \times \Z^d)$ are defined via 
\begin{align} 
  \alpha_t\big(\xi, \{ (\xi,z) \} \big) = \frac{\xi_z p_{z,x_0}(T-t)}{h(\xi,t)} 
  = \xi_z s_z(\xi, t), \quad \xi \in \mathbb{S}_w, \: z \in \Z^d
\end{align}
with $h(\xi,t) = \sum_x \xi_x p_{x,x_0}(T-t)$ from \eqref{eq:h.eta.t}. Obviously 
$\alpha_t\big(\xi, \pi_{\mathbb{S}_w}^{-1}(\{\xi\})\big) = 1$ for each $\xi$. 
\smallskip

We can view this as a
``filtering problem'' for the process with a tagged site and
\eqref{eq:projtildeexi=hateta} is a consequence of (a
time-inhomogeneous version of) a Markov mapping theorem, see e.g.\
\cite[Thm.~A.15]{KR11} or \cite[Cor.~3.3]{KN11}.  Note that these
results are literally formulated for time-homogeneous Markov processes,
but the time-inhomogeneous case can be easily included by considering
time as an additional coordinate of the process.  (For the function
$\psi$ in \cite[Thm.~A.15]{KR11}/\cite[Cor.~3.3]{KN11} we can use $\psi_w(\eta,z) := 1 +
||\eta||_w$ with $||\eta||_w$ from \eqref{def:wnorm}). 
\smallskip

Consider suitable test functions $f : \mathbb{S}_w \times \Z^d \times [0,T] \to \R$, 
define a function $g$ ($=g(f)$) on $\mathbb{S}_w$ via 
\begin{align} 
g(\eta, t) := \int_{\mathbb{S} \times \Z^d} f(\xi,z,t) \, \alpha_t\big(\eta, d(\xi,z)\big). 
\end{align}
Note that $\pi_{\mathbb{S}_w}(\widetilde\xi(0), \widetilde{X}(0)) = \widetilde\xi(0) \mathop{=}^d 
\widehat{\eta}(0)$ by construction. 
To conclude \eqref{eq:projtildeexi=hateta} for $T>0$ we need to verify that 
\begin{align} 
  \label{eq:intLtilde=Lhat}
  \int_{\mathbb{S}_w \times \Z^d} \widetilde{L}_t f(\xi,z,t) \, \alpha_t\big(\eta, d(\xi,z)\big) 
   = \widehat{L}_t g(\eta,t ).
\end{align}
It suffices to consider functions $f$ of the form 
\begin{align} 
  \label{eq:projtildeexi=hateta.f-form}
  f(\xi, z, t) = f_1(\xi,t) 1_{z=z_0}
\end{align}
for some suitable $f_1 : \mathbb{S}_w \times [0,T] \to \R$ and $z_0 \in \Z^d$. 
The proof that \eqref{eq:intLtilde=Lhat} holds for such functions is a lengthy but straightforward 
computation with the generators and is delegated to Appendix~\ref{sect:DoobTransform-aux}.

\subsection{The size-biased process viewed from the immigration source} \label{subsect:relviewpoint}
We have just shown that the locally size-biased process $(\widehat \eta^{(x,T)})_{0\leq t\leq T}$, together with the randomly chosen path $X$, has the same distribution as $(\widetilde \xi^{(x,T)}, X)$, where $X$ can be interpreted as the immigration source. When $\eta(0)$
is translation invariant, it is easily seen that the process $\xi=(\xi_z(t))_{z\in\Z^d, 0 \le t \le T}$ with $\xi_z(t):= \widetilde \xi^{(x,T)}_{X(t)+z}(t)-\delta_{0}$, where the immigration source is shifted to the origin and removed from the particle configuration, is a time-homogeneous Markov process with (formal) generator 
\[
L = L_\mathrm{rw} + L_\mathrm{br} + L_\mathrm{im} + L_\mathrm{mf}
\]
with 
\begin{equation}\label{xigen}
\begin{aligned}
L_\mathrm{rw}f(\xi) &= \sum_{x, y} \xi_x p_{x y} 
        \big( f(\xi^{x \to y}) - f(\xi) \big), \\
L_\mathrm{br} f(\xi) & = \gamma \sum_{x \neq 0} 1_{\{\eta_x=1\}} \frac12 
                        \big( f(\xi^{+x}) +  f(\xi^{-x})- 2f(\xi) \big), \\
L_\mathrm{im} f(\xi) & = \gamma 1_{\{\xi_0=0\}} \big( f(\xi + \delta_0) - f(\xi) \big), \\ 
L_\mathrm{mf} f(\xi) & =  \sum_x p_x \big( f(\theta_x \xi) - f(\xi) \big),
\end{aligned}
\end{equation}
which encode respectively the random walk motions of the particles, the lonely critical binary branching of the particles, the immigration of particles at the origin, and the spatial shift $(\theta_x \xi)_y = \xi_{x+y}$ to compensate the jumps of the immigration source. 

The process $\xi$ is clearly a well-defined Markov process on the space of finite configurations
\begin{align} \label{Sfin}
  \mathbb{S}_{\rm fin} := \{ \xi \in {\mathbb N}_0^{\Z^d} \, : \, 
  {\textstyle\sum_{x\in\Z^d}} \xi_x< \infty \}.
\end{align}
Let us equip $\mathbb{S}_{\rm fin}$ with the partial order $\preceq$ such that $\xi\preceq \xi'$ if and only if $\xi_x\leq \xi'_x$ for all $x\in\Z^d$. It is then easily seen that $\xi$ is monotone in the sense that: given two initial configurations $\xi(0)\preceq \xi'(0)$, there is a coupling such that almost surely, $\xi(t)\preceq \xi'(t)$ for all $t\geq 0$. For this, one can use, for example, a small adaptation of the construction in \cite[Section~2.2]{B03}.

Using this monotonicity, we can further extend the state space of $\xi$ to 
\begin{equation}
\mathbb{S}:=(\N_0\cup\{\infty\})^{\Z^d}, 
\end{equation}
equipped with the same partial order $\preceq$. More precisely, for any $\xi(0)\in \mathbb{S}$, let $\xi^{(n)}(0)\in \mathbb{S}_{\rm fin}$ be any sequence which increases monotonically to $\xi(0)$. We then define $(\xi(t))_{t\geq 0}$ to be the monotone limit of $(\xi^{(n)}_t)_{t\geq 0}$ under the afore-mentioned coupling of $(\xi^{(n)})_{n\in\N}$. Note that the law of $(\xi_t)_{t\geq 0}$ does not depend on the choice of $\xi^{(n)}(0)\uparrow \xi(0)$. 
It is in principle possible that $\xi_x(0)$ grows so quickly as $|x|\to\infty$ that 
$\xi_{x'}(t') = \infty$ occurs at some point $t' \ge 0$ for some $x'$ and then $\xi_\cdot(t'') \equiv +\infty$ 
for all $t''>t'$; however, this will not be the case for the initial conditions we consider below.
\smallskip

Inspection of the construction of $\widetilde \xi^{(x,T)}$ and its relation with $\xi$ shows that: if $\E[\eta_y(0)]$ is constant in $y\in\Z^d$ 
(which we can assume by the remark after Theorem~\ref{thm1}),  then the shifted path $(X(t)-X(0))_{0\leq t\leq T}$ from \eqref{X} is a random walk with transition kernel $p$, and for any $T>0$, we have the stochastic domination relation
\begin{equation}\label{domination}
\mathscr{L}(\xi(T) \, | \xi(0) \equiv 0) \preceq 
\mathscr{L}(\widetilde{\xi}^{(0,T)}(T)) = \mathscr{L}(\widehat{\eta}^{(0,T)}(T))
\end{equation}
(we can think of $\xi$ as describing 
a subset of the particles in $\widetilde{\xi}^{(0,T)}$, namely only 
the relatives of the selected particle). To prove \eqref{eq:sbdiv} and conclude the proof of Theorem~\ref{thm1}, it then suffices to show that given $\xi(0)\equiv 0$, $\xi_x(t)\to\infty$ in probability for all $x\in\Z^d$.

\section{Proof of Theorem~\ref{thm1}}
\label{sect:proofthm1}

As noted after \eqref{domination}, to prove Theorem \ref{thm1}, it suffices to show that $\xi$, the locally size-biased process viewed from the immigration source introduced in Section \ref{subsect:relviewpoint} above, diverges locally with probability 1. We will accomplish this by first establishing a dichotomy between $(\xi_x(t))_{t\geq 0}$ being tight and $\xi_x(t)\to\infty$ in probability for every $x\in\Z^d$, formulated in Lemma \ref{lem:xidichot} below. We will then rule out tightness by contradiction, using first and second moment bounds for $\xi$ and the Paley-Zygmund inequality. 

%

\subsection{Dichotomy between tightness and unbounded growth}
\label{subsect:dichotomy}


\begin{lem} 
  \label{lem:xidichot}
  The process $\xi:=(\xi_t)_{t\geq 0}$ is monotone on the state space $\mathbb{S}$. Furthermore, starting from
  $\xi(0) \equiv 0$, the law $\mathcal{L}(\xi(t))$ is stochastically
  non-decreasing in $t$, and the following dichotomy holds: 
 \begin{itemize}
 \item[\rm i)] either $\big\{{\cal L}(\xi_x(t)) :{t \ge 0} \big\}$ is tight for 
  every $x\in\Z^d$;  
  \item[\rm ii)] or $\xi_x(t) \to \infty$ in probability as $t \to \infty$ for every $x\in\Z^d$. 
  \end{itemize}
  In case {\rm i)}, we have $\xi(t) \Rightarrow \xi^{(\infty)}\in \N_0^{\Z^d}$ in the sense of finite-dimensional 
  distributions, 
  where ${\cal L}(\xi^{(\infty)})$ is a stationary law for the process, with $\P(\xi^{(\infty)}_0=0)>0$.
\end{lem}
\begin{proof} The monotonicity of $\xi$ on the state space $\mathbb{S}$ is inherited from its monotonicity 
on the space of finite configurations, $\mathbb{S}_{\rm fin}$, defined in \eqref{Sfin}. Given $\xi(0)\equiv 0$, we have $\xi(0)\preceq \xi(s)$ for any $s\geq 0$. It then follows that the law of $\xi(t)$ is stochastically non-decreasing in $t\geq 0$, and as $t\to\infty$, $\xi(t)$ converges in finite-dimensional distribution to a limit $\xi^{(\infty)}\in \mathbb{S}$.

We first assume i), that $\big\{{\cal L}(\xi_x(t)) :{t \ge 0} \big\}$ is tight for every $x\in\Z^d$. Then $\xi^{(\infty)}\in \N_0^{\Z^d}$ almost surely. We claim that the law ${\cal L}(\xi^{(\infty)})$ is stationary for the process $\xi$. Indeed, let $\xi'$ start with $\xi'(0)=\xi^{(\infty)}$. For any $s, t>0$, since ${\cal L}(\xi(s)) \preceq {\cal L}(\xi'(0))$, we must have ${\cal L}(\xi(s+t)) \preceq {\cal L}(\xi'(t))$. Letting $s\to\infty$ then shows that ${\cal L}(\xi^{(\infty)}) \preceq {\cal L}(\xi'(t))$. On the other hand, $\xi'(t)$ can be constructed as the monotone limit of $\xi^{'(n)}(t)$ with initial condition $\xi^{'(n)}_x(0):=\xi_x(n) 1_{\{|x|\leq n\}}$, where $\xi^{'(n)}(0) \in \mathbb{S}_{\rm fin}$ and $\xi^{'(n)}(0) \uparrow \xi'(0)=\xi^{(\infty)}$ under a suitable coupling of $(\xi(n))_{n\in\N}$ and $\xi^{(\infty)}$. Note that for all $n\in\N$, ${\cal L}(\xi^{'(n)}(t)) \preceq {\cal L}(\xi(n+t)) \preceq {\cal L}(\xi^{(\infty)})$. It then follows that ${\cal L}(\xi'(t))\preceq {\cal L}(\xi^{(\infty)})$. Therefore ${\cal L}(\xi'(t))={\cal L}(\xi^{(\infty)})$ for all $t\geq 0$, and ${\cal L}(\xi^{(\infty)})$ is a stationary law for $\xi$.
\medskip

In order to show that if i) fails, ii) must hold, we use monotonicity and a simple ``re-start'' argument. 
One can alternatively prove that claim via an 
explicit, though lengthy to formulate, coupling construction and the Hewitt-Savage-0-1-law, analogous 
to \cite[Sect.~3.2]{B03}. 

Let us now assume that i) fails, so that $\{{\cal L}(\xi_x(t)): t\geq 0\}$ is not tight for some $x\in \Z^d$. Then $\P(\xi^{(\infty)}_x=\infty)=\eps$
for some $\eps\in (0,1]$. Since for any $y\in\Z^d$, there is a fixed positive probability that a particle from $x$ will move to $y$ in unit time without undergoing any branching or death, we conclude that we must have $\P(\xi^{(\infty)}_y=\infty)\geq \eps$ for all $y\in\Z^d$. Switching $x$ and $y$ then shows that $\P(\xi^{(\infty)}_x=\infty)= \eps$ for all $x\in\Z^d$. We will prove $\eps=1$ by contradiction.

First note that since $\xi_0(t)$ converges in law to $\xi_0^{(\infty)}$, for any $\delta>0$ and $K>0$, we have 
\begin{equation}\label{tight1}
\P(\xi_0(t) >K) \geq \eps-\delta \qquad \mbox{for all $t$ large enough}.
\end{equation}

Let $\xi^{'(n)}$ be a sequence of the $\xi$ process with initial condition $\xi^{'(n)}(0)=\xi(n)$, coupled in such a way that almost surely, $\xi^{'(n)}(0)\uparrow \xi^{(\infty)}$. Conditioned on a sequence of initial conditions $\xi^{'(n)}(0)$ satisfying $\xi^{'(n)}_0(0)\uparrow \xi_0^{(\infty)}<\infty$, which occurs with probability $1-\eps$, by monotonicity, we have ${\cal L}(\xi(t))\preceq {\cal L}(\xi^{'(n)}(t) | \xi^{'(n)}(0))$ for all $n\in\N$ and $t>0$. In particular, by \eqref{tight1}, we can choose $t$ large enough such that uniformly in $n\in\N$ and $\xi^{'(n)}(0)$,
$$
\P(\xi^{'(n)}_0(t)\geq K |\xi^{'(n)}(0))\geq \eps-\delta.
$$
On the other hand, conditioned on a sequence of initial conditions $\xi^{'(n)}(0)$ satisfying $\xi^{'(n)}_0(0)\uparrow \xi_0^{(\infty)}=\infty$, which occurs with probability $\eps$, we have $\xi^{'(n)}_0(t)\to \infty$ as $n\to\infty$ in probability, since there is a fixed probability for a particle to start from the origin and return to the origin at time $t$ without undergoing branching or death along the way. Combining the above two cases, we conclude that for all $n$ large enough, 
$$
\P(\xi^{'(n)}_0(t)\geq K) = \P(\xi_0(n+t)\geq K) \geq (1-\eps) (\eps-\delta) +\eps(1-\delta). 
$$
In particular, $\P(\xi^{(\infty)}_0 \geq K) \geq (1-\eps) (\eps-\delta) +\eps(1-\delta)>\eps$ if $\eps\in (0,1)$ and $\delta$ is chosen sufficiently small. Since $K$ can be chosen arbitrarily large, this implies that $\P(\xi^{(\infty)}_0 =\infty) >\eps$, which is a contradiction. Therefore when i) fails, we must have $\eps=1$, i.e., $\xi_x(t)\to\infty$ in probability for all $x\in\Z^d$.
\medskip

Lastly, we show that in case i), $\P(\xi^{(\infty)}_0=0)>0$. Recall that $p_{xy}(t)$ denotes the transition probability kernel of a random walk with jump kernel $p$. First we claim that: 
\begin{equation}\label{xinfty}
\mbox{For all } t>0, \quad \sum_z p_{z0}(t) \xi^{(\infty)}_z <\infty \qquad \mbox{almost surely}. 
\end{equation}
Let us consider the stationary process $\xi'$ with $\xi'(0)=\xi^{(\infty)}$. If \eqref{xinfty} fails, then for some $t_0>0$, $\sum_z p_{z0}(t_0) \xi^{(\infty)}_z =\infty$ with positive probability. Let us fix an initial configuration $\xi'(0)$ with $\sum_z p_{z0}(t_0) \xi'_z(0) =\infty$. With probability $e^{-(1+\gamma)}$, the immigration source $X$ in the locally size-biased system $\widetilde \xi'$ does not move and has no immigration during the time interval $[0,1]$. Conditioned on this event, we have $\xi'(t)=\widetilde \xi'(t) -\delta_0$ for $t\in [0,1]$, and the $\xi'$ system is easily seen to stochastically dominate a collection of independent random walks $\xi''$ with initial condition $\xi''(0):=\xi'(0)$, where each walk jumps with rate $1$ and kernel $p$ and dies with rate $\gamma$, regardless of whether it is alone or not. A Borel-Cantelli argument then shows that given $\sum_z p_{z0}(t_0) \xi''_z(0) =\infty$, we must have $\xi''_0(t_0)=\infty$ a.s., and hence $\xi'_0(t_0)=\infty$ a.s.. It follows that $\P(\xi'_0(t_0)=\infty)=\P(\xi^{(\infty)}_0=\infty)>0$, which is a contradiction. Therefore \eqref{xinfty} must hold.
\smallskip

Given $\xi'(0)$ with $\sum_z p_{z0}(1) \xi'_z(0) <\infty$, we now show that $\P(\xi'_0(1)=0 | \xi'(0))>0$, which implies $\P(\xi^{(\infty)}_0=0)>0$ by the stationarity of $\xi'$. Again, let us restrict to the event that the immigration source does not move or have immigration during the time interval $[0,1]$. Conditioned on this event, the $\xi$ is system is easily seen to be stochastically dominated by a collection of independent branching random walks $\xi'''$ with initial condition $\xi'''(0):=\xi'(0)$, where each walk jumps with rate $1$ and kernel $p$ and branches into two with rate $\gamma$. We can choose $L$ large enough such that the expected number of particles that originate from outside $[-L,L]$ at time 0 and is at $0$ at time $1$, is less than $1$, so that with positive probability, no particle originating from outside $[-L,L]$ will be at the origin at time $1$. Clearly there is also positive probability that none of the particles originating from $[-L,L]$ will have an offspring at the origin at time $1$. Therefore we have $\P(\xi'''_0(1)=0 | \xi'''(0))>0$, and the same holds for $\xi'$. 
\end{proof}

\subsection{Moment computations for $\xi$}
\label{subsect:momentcomp}
We now derive bounds on the first and second moments of $\xi_x(t)$. 
Note that we require the results discussed in this section only for $\xi(0) \in \mathbb{S}_\mathrm{fin}$
(in fact, only for $\xi_\cdot(0) \equiv 0$), so that $\xi(t) \in \mathbb{S}_\mathrm{fin}$ for all $t\ge 0$ 
and the expressions involving the generator will always be well-defined.
\smallskip

To keep track of the joint positions of two particles in the $\xi$ system, we introduce two dependent random walks $(\widehat{X})_{t\geq 0}$ and $(\widehat{X}'(t))_{t \ge 0}$ on $\Z^d$, such that
\[
(\widehat{X}(t), \widehat{X}'(t)) = \big( \widehat Y(t)-\widehat Y^{(0)}(t), 
\widehat Y'(t)-\widehat Y^{(0)}(t) \big)
\]
where $\widehat Y^{(0)}, \widehat Y, \widehat Y'$ are three
independent random walks with jump rate $1$ and jump kernel $(p_{z})_{z\in\Z^d}$.  The walks $\widehat Y$ and $\widehat Y'$ represent the independent motions of two particles in the $\widetilde \xi$ system, which is the stochastic representation of the locally size-biased branching random walks with a moving immigration source, while $\widehat Y^{(0)}$ represents the motion of the immigration source in $\widetilde \xi$. 

Note that individually, both $(\widehat{X}(t))_{t \ge 0}$ and $(\widehat{X}'(t))_{t \ge 0}$ are random walks with jump rate $2$ and jump kernel $(\frac{p_{-z}+p_z}{2})_{z\in\Z^d}$. Its generator is given by
\[
(\widehat{L}^{(1)} f)_x := \sum_z (p_{-z}+p_z)(f_{x+z}-f_x).
\] 
Let $\widehat{p}_{xy}(t) := \P_x(\widehat{X}(t)=y)$ denote its transition probability kernel, with $\widehat{p}_0:=\widehat{p}_{00}$. Let $\widehat{L}^{(1),*}$ denote the generator of
the time-reversed random walk for $\widehat X$, which has the same distribution as $-\widehat{X}$, with transition kernel $\widehat{q}_{xy}(t) := \widehat{p}_{yx}(t) = \widehat{p}_{xy}(t)$ by symmetry. 

Note that jointly $(\widehat{X}(t), \widehat{X}'(t))_{t \ge 0}$ is a random walk on $\Z^{2d}$ with generator 
\[
(\widehat{L}^{(2)}f)_{x,y} := \sum_z p_{z} (f_{x+z,y} + f_{x,y+z} +
f_{x-z,y-z} - 3f_{x,y}).
\]
Let 
\begin{equation} 
  \label{eq:phat2t}
  \widehat{p}^{(2)}_{(x,y),(w,z)}(t) := \P_{(x,y)}\big((\widehat{X}(t), \widehat{X}'(t))=(w,z)\big)
\end{equation} 
denote its transition probability kernel. 
Let $\widehat{L}^{(2),*}$ denote the generator for the time-reversal of $(\widehat{X}, \widehat{X}')$, which has the same distribution as 
$(-\widehat{X}, -\widehat{X}')$, with transition kernel $\widehat{q}^{(2)}_{(x,y), (w,z)}(t) := \widehat{p}^{(2)}_{(w,z), (x,y)}(t)$.
\smallskip

\begin{lem} 
\label{lem:twomoments} The first two moments of  $\xi_\cdot(t)$ admit the following representation:
\begin{itemize}
\item[\rm (1)]  Assume that $\E[\xi_y(0)] < \infty$ for all $y\in\Z^d$. Then for $t\ge 0$, $x\in\Z^d$
\begin{align} 
  \label{eq:mean1}
  \E[\xi_x(t)] & = \sum_y \P_y\big( \widehat{X}(t) = x \big)
  \E[\xi_y(0)]
  + \gamma \int_0^t \P_0\big( \widehat{X}(t-s) = x \big) \P(\xi_0(s) = 0) \, ds \\
  \label{eq:mean2}
  & = \sum_y \E[\xi_y(0)] \widehat{p}_{yx}(t) + \gamma \int_0^t \P(\xi_0(s) =
  0) \widehat{p}_{0x}(t-s) \, ds.
\end{align}

\item[\rm (2)]  Assume that $\E[\xi_y(0)^2] < \infty$ for all $y\in\Z^d$. Then for $t\ge 0$, $x, y\in\Z^d$
\begin{align} 
  \E \big[ & \xi_x(t)(\xi_y(t) - \delta_{xy})\big] \notag \\
  = & \sum_{x',y'} \P_{(x',y')}\big(
  (\widehat{X}(t),\widehat{X}'(t)) = (x,y) \big)
  \E\big[\xi_{x'}(0)(\xi_{y'}(0) - \delta_{x'y'})\big] \notag \\
  & + \gamma \int_0^t \sum_{z'\neq 0} \P_{(z',z')}\big(
  (\widehat{X}(t-s),\widehat{X}'(t-s)) = (x,y) \big)
  \P(\xi_{z'}(s) = 1) \, ds \notag \\
  & + \gamma \int_0^t \sum_{y'\neq 0} \P_{(0,y')}\big(
  (\widehat{X}(t-s),\widehat{X}'(t-s)) = (x,y) \big) 
  \E\big[ 1_{\{\xi_0(s)=0\}} \xi_{y'}(s) \big] \, ds \notag \\
  \label{eq:2ndfmom}
  & + \gamma \int_0^t \sum_{x'\neq 0} \P_{(x',0)}\big(
  (\widehat{X}(t-s),\widehat{X}'(t-s)) = (x,y) \big) 
  \E\big[ 1_{\{\xi_0(s)=0\}} \xi_{x'}(s) \big] \, ds  \\
  = & \sum_{x',y'} \E\big[\xi_{x'}(0)(\xi_{y'}(0) - \delta_{x'y'})\big]
  \widehat{p}^{(2)}_{(x',y'),(x,y)}(t) 
  + \gamma \int_0^t \sum_{z\neq 0} \P(\xi_z(s) = 1) \widehat{p}^{(2)}_{(z,z),(x,y)}(t-s) \, ds
  \notag \\
  & + \gamma \int_0^t \sum_{z\neq 0} \E\big[ 1_{\{\xi_0(s)=0\}} \xi_{z}(s) \big] 
  \big\{ \widehat{p}^{(2)}_{(0,z),(x,y)}(t-s) + \widehat{p}^{(2)}_{(z,0),(x,y)}(t-s)\big\} \, ds .
  \label{eq:2ndfmom1}
\end{align}
\end{itemize}
\end{lem}
\br\rm 
Note that $\xi_x(t)(\xi_y(t) - \delta_{xy})$ counts the number of pairs of particles, with the first particle from position $x$ and the second from position $y$ at time $t$. The terms in the sum in \eqref{eq:2ndfmom} are respectively contributions from the following cases: the pair of particles sampled from $x$ and $y$ at time $t$ come from distinct ancestors at time 0; the pair of particles come from the same ancestor; the pair of particles come from distinct ancestors with at least one ancestor being a particle added at the immigration source at the origin. 
\er
\begin{proof} 
  \noindent (1) Let $f_x(t) := \E\big[ \xi_x(t) \big]$. It is easily seen that (cf.\ \eqref{eq:geneqmom1} below) $f$ solves 
  \begin{align}\label{duhamel1}
    \frac{\partial}{\partial t}f_x(t) = 
    \big(\widehat{L}^{(1), *}f_\cdot(t)\big)_x + \gamma  \delta_{x0} \P(\xi_0(t)=0).
  \end{align}
  Applying Duhamel's principle for semilinear equations (e.g.\ \cite[Thm.~6.1.2]{P83}) and using the fact that the random walk with generator $\widehat{L}^{(1), *}$ has the same distribution as the time reversal of $\widehat X$, we obtain \eqref{eq:mean1}.  
  \smallskip

  \noindent (2) Let $f_{x,y}(t):= \E\big[\xi_x(t)(\xi_y(t)-\delta_{xy})\big]$, which is easily seen to solve (cf.\ \eqref{eq:geneqmom2} below)
    \begin{equation}\label{duhamel2}
    \begin{aligned}
  \frac{\partial}{\partial t}f_{x,y}(t) 
  = & \big(\widehat{L}^{(2), *}f_{\cdot,\cdot}(t) \big)_{x,y} 
  + \gamma \delta_{xy}(1-\delta_{x0}) \P(\xi_x(t)=1) \\
  & + \gamma \Big( \delta_{x0} \E\big[ 1_{\{\xi_0(t)=0\}} \xi_y(t) \big] 
  + \delta_{y0} \E\big[ 1_{\{\xi_0(t)=0\}} \xi_x(t) \big] \Big). 
  \end{aligned}
  \end{equation}
  Again, applying Duhamel's principle and using the fact that the random walk with generator $\widehat{L}^{(2), *}$ has the same distribution as the time reversal of $(\widehat X, \widehat X')$,  we obtain \eqref{eq:2ndfmom}.
\end{proof}

Using Lemma \ref{lem:twomoments}, we now bound the first two moments of $\xi_\cdot(t)$.
\begin{lem} 
\label{lem:2ndmom}
Let $\xi_\cdot(0) \equiv 0$. We have 
  \begin{align}
    \label{eq:bdmom1}
    & \E\big[\xi_x(t)\big] \le \,
    \gamma \int_0^t 
    \widehat{p}_{0x}(u) \, du, \\
    \label{eq:bdmom2}
    & \E\big[\xi_x(t)( \xi_y(t) - \delta_{xy}) \big] \notag \\ 
    & \hspace{1em} \le \gamma^2 \int_0^t \int_0^s \sum_{z\neq 0} \widehat{p}_{0z}(u) 
    \big\{ \widehat{p}^{(2)}_{(z,z),(x,y)}(t-s) + \widehat{p}^{(2)}_{(0,z),(x,y)}(t-s) 
    + \widehat{p}^{(2)}_{(z,0),(x,y)}(t-s) \big\} \, du \, ds.
  \end{align}
\end{lem}
\begin{proof} Note that \eqref{eq:bdmom1} follows from \eqref{eq:mean2} in Lemma~\ref{lem:twomoments},
using the trivial bound $\P(\xi_0(s) = 0) \le 1$.

To verify \eqref{eq:bdmom2}, we insert the bounds 
$\P(\xi_{z'}(s) = 1) \le \E[\xi_{z'}(s)]$, 
$\E\big[ 1_{\{\xi_0(s)=0\}} \xi_{y'}(s) \big] \le \E[\xi_{y'}(s)]$, 
$\E\big[ 1_{\{\xi_0(s)=0\}} \xi_{x'}(s) \big] \le \E[\xi_{x'}(s)]$, 
together with \eqref{eq:bdmom1}, into \eqref{eq:2ndfmom1} in
Lemma~\ref{lem:twomoments}. 
\end{proof}

\begin{lem} 
  \label{lem_2ndmombound} 
  We have 
  \begin{align}
    \label{eq:2ndmombound}
    \E\big[ (\xi_0(t))^2 \,\big|\, \xi_\cdot(0) \equiv 0 \big] \le
    3\gamma^2 \Big( \int_0^t \widehat{p}_0(s) \,ds \Big)^2 +
    \gamma \int_0^t \widehat{p}_0(s) \,ds
    \quad \mbox{for all } t>0. 
  \end{align}
\end{lem}
\begin{proof} Recall the following well-known fact about symmetric, continuous-time
  random walks:
  \begin{align}
    \label{eq:RWtransfact}
    \widehat{p}_{0,z}(v) \le \widehat{p}_0(v) \le \widehat{p}_0(u)
     \quad \text{for all } z \in \Z^d, \, 0 \le u \le v < \infty.
  \end{align}
  For completeness and lack of a point reference, this follows from
  Fourier inversion: For $k \in [0,2\pi)$ let
  $\varphi(k) := \sum_{x \in \Z} e^{i k x} (p_x+p_{-x})/2 \in [-1,1]$
  be the characteristic function of the jump distribution of
  $\widehat{X}$, then
  $\varphi_t(k) := \E_0[e^{i k \widehat{X}(t)}] = \exp\big(-2t (1-\varphi(k))\big) \in [0,1]$
  and \eqref{eq:RWtransfact} follows from
  $\widehat{p}_{0,z}(t) = \frac{1}{2\pi} \int_0^{2\pi} e^{-i k z} \varphi_t(k) \, dk
  = \frac{1}{2\pi} \int_0^{2\pi} \cos(k z) \varphi_t(k) \, dk$.
  \medskip
  
  Recalling the definition of $\widehat{p}^{(2)}$ from \eqref{eq:phat2t}
  and using the second inequality in \eqref{eq:RWtransfact} in the second line, we find 
  \begin{align*}
    \sum_{z \neq 0} \widehat{p}^{(2)}_{(z,z),(0,0)}(v) 
    &
      \le \sum_z \widehat{p}^{(2)}_{(z,z),(0,0)}(v)
      = \sum_z \widehat{q}^{(2)}_{(0,0),(z,z)}(v) 
      = \P\big( -\widehat{X}(v) = -\widehat{X}'(v) \big) \\
    & = \P\big( \widehat{Y}(v) - \widehat{Y}'(v) = 0 \big) 
      = \widehat{p}_0(2v) \le \widehat{p}_0(v),
  \end{align*}
  \begin{align*}
    \sum_{z \neq 0} \big( \widehat{p}^{(2)}_{(0,z),(0,0)}(v) + \widehat{p}^{(2)}_{(z,0),(0,0)}(v) \big) 
    \le \sum_z \big( \widehat{p}^{(2)}_{(0,0),(z,0)}(v) + \widehat{p}^{(2)}_{(0,0),(0,z)}(v) \big)
    = 2 \widehat{p}_0(v)
  \end{align*}
  Using this and the first inequality in \eqref{eq:RWtransfact}, we
  can bound \eqref{eq:bdmom2} from Lemma~\ref{lem:2ndmom} for $x=y=0$
  as follows:
  \begin{align*}
    \E\big[\xi_0(t)( \xi_0(t) - 1) \big]
    \le 3 \gamma^2 \int_0^t \int_0^s \widehat{p}_0(u) \widehat{p}_0(t-s) \, du \,ds
    \le 3 \gamma^2 \int_0^t \int_0^t \widehat{p}_0(u) \, \widehat{p}_0(v) \, du \,dv.
  \end{align*}
  Combining with \eqref{eq:bdmom1} yields \eqref{eq:2ndmombound}.

\end{proof}

\subsection{Long-time behaviour of $\xi$}
We now conclude the proof of Theorem~\ref{thm1} by ruling out tightness in Lemma \ref{lem:xidichot}.

\begin{lem} 
  \label{lem:xidiv}
  If $p$ satisfies Assumption~\ref{ass:p}, then starting from $\xi_\cdot(0) \equiv 0$, we have $\xi_x(t) \to
  \infty$ in probability as $t \to \infty$ for any $x \in \Z^d$.
\end{lem}

\begin{proof} 
By Lemma~\ref{lem:xidichot}, it suffices to show that the family $(\xi_0(t))_{t \ge 0}$ is not tight.
We argue by contradiction: Assume that this is not the case, 
then we obtain from Lemma~\ref{lem:xidichot} that $(\xi_x(t))_{t\geq 0}$ must be tight for every $x\in\Z^d$, and 
$\xi(t)$ converges in distribution to a non-trivial equilibrium $\xi^{(\infty)}\in \N_0^{\Z^d}$. In particular, we have 
\begin{align} 
\label{eq:nonoccproplim}
\lim_{t\to\infty} \P(\xi_0(t)=0) = \P(\xi^{(\infty)}_0=0) =:b > 0.
\end{align}

Straightforward computation 
using \eqref{eq:nonoccproplim} (and \eqref{eq:mean1} in 
Lemma~\ref{lem:twomoments}) then yields 
\[ \E[\xi_0(t)] \sim b\gamma \int_0^t \widehat{p}_0(s) \,ds \to \infty
\quad \mbox{as } t\to\infty.
\] 
Combined with \eqref{eq:2ndmombound} from Lemma~\ref{lem_2ndmombound} and applying the Paley-Zygmund inequality,
we have
\begin{align} 
  \label{eq:PaleyZygmundbd}
  \inf_{t \ge 0} \P\big( \xi_0(t) \ge \tfrac12 \E[\xi_0(t)] \big) 
  \ge \inf_{t \ge 0} \frac14 \cdot \frac{\E[\xi_0(t)]^2}{\E[\xi_0(t)^2]} 
  > 0.
\end{align}
It follows that $(\xi_0(t))_{t\geq 0}$ is not tight because our assumption implies $\E[\xi_0(t)] \to \infty$, which contradicts the assumption that $(\xi_0(t))_{t\geq 0}$ is tight. Therefore $(\xi_0(t))_{t\geq 0}$ cannot be tight. 
\end{proof}

\begin{rmk} \rm 
  \label{rem:p.req.prop}


  A natural generalisation of the lonely lonely branching random
  walks is to consider SCBRW$_b$ (as defined in
  Section~\ref{sect:discussion}) with branching rate function $b(j) =
  \gamma 1_{j=j_*}$ for some $j_* \ge 2$ and $\gamma>0$.  It turns out
  that the arguments from Sections\ \ref{sect:Palm} and
  \ref{subsect:momentcomp} can be adapted in a fairly straightforward
  way to this case.  However, it seems not obvious how to then obtain
  the dichotomy between tightness and growth as in
  Section~\ref{subsect:dichotomy}.  Obviously, one could now not
  simply start the $\xi$ system from the empty configuration and
  starting from some other initial condition it is not a priori clear
  how to implement a restart argument.

  We believe that a suitable analogue of Theorem~\ref{thm1} holds 
  but we defer this to future research.
\end{rmk}

\begin{appendix} 
\bigskip

\noindent {\bf \Large Appendix}

\section{Generator computations for the moments}
For completeness, we include here the generator calculations used in the proof of Lemma \ref{lem:twomoments}.

Recall the different components of the generator for $\xi$ from \eqref{xigen}. To derive \eqref{duhamel1} for $f_x(t):= \E[\xi_x(t)]$, let $F_x(\xi) := \xi_x$. We have 
  \begin{align*} 
    \big(L_{rw} + L_{mf}\big)F_x(\xi) & = \Big(\sum_y \xi_y p_{yx} - \xi_x \Big) 
    + \sum_z p_z \big(\xi_{x+z} - \xi_x \big) \\
    & = \sum_z (p_z+p_{-z}) \big( \xi_{x+z} - \xi_x) = \big(\widehat{L}^{(1), *} F_\cdot(\xi)\big)_x, \\
    L_{im} F_x(\xi) & = \gamma \delta_{x0}  1_{\{\xi_0=0\}}, \quad 
    L_{br} F_x(\xi) = 0,
  \end{align*}
  hence 
  \begin{align} 
    \label{eq:geneqmom1}
    L F_x(\xi) & = \big(\widehat{L}^{(1), *}F_\cdot(\xi)\big)_x + \delta_{x0} \gamma 1_{\{\xi_0=0\}},
  \end{align}
  which implies that $f_x(t)=\E[F_x(\xi(t))]$ satisfies the equation \eqref{duhamel1}.  
  \smallskip

To derive \eqref{duhamel2} for $f_{x,y}(t):= \E\big[\xi_x(t)(\xi_y(t)-\delta_{xy})\big]$, let $F_{x,y}(\xi) := \xi_x(\xi_y-\delta_{xy})$. We have 
  \begin{align*} 
    L_{rw} F_{x,y}(\xi) & = \sum_{v,w} \xi_v p_{vw} \Big[ \big( \xi_x + \delta_{xw}-\delta_{xv}\big) 
    \big(\xi_y - \delta_{xy} + \delta_{yw} - \delta_{yv} \big) 
    - \xi_x(\xi_y - \delta_{xy}) \Big] \\
    & = \sum_{v,w} \xi_v p_{vw} \Big[ \xi_x(\delta_{yw} - \delta_{yv}) + \xi_y(\delta_{xw}-\delta_{xv}) 
    + (\delta_{xw}-\delta_{xv}) (\delta_{yw}-\delta_{yv} - \delta_{xy}) \Big] \\
    & = \xi_x \sum_v \xi_v p_{vy} - \xi_x \sum_w \xi_y p_{vy} 
    + \xi_y \sum_v \xi_v p_{vx} -\xi_y \sum_w \xi_x p_{xw} \\
    & \hspace{1.5em} + \delta_{xy} \sum_v \xi_v p_{vx} - \xi_y p_{yx} 
    - \delta_{xy} \sum_v \xi_v p_{vx} 
    - \xi_x p_{xy} + \delta_{xy} \sum_w \xi_x p_{xw} + \delta_{xy} \sum_w \xi_x p_{xw} \\
    & = \sum_v p_{vy} \big( \xi_x \xi_v - \xi_x\xi_y \big) 
    + \sum_v p_{vx} \big( \xi_v \xi_y - \xi_x\xi_y \big)
    - \xi_y p_{yx} - \xi_x p_{xy} + 2 \delta_{xy} \xi_x \\
    & = \sum_v p_{vy} \big( \xi_x (\xi_v - \delta_{vx}) - \xi_x(\xi_y-\delta_{xy}) \big) 
    + \sum_v p_{vx} \big( \xi_v (\xi_y - \delta_{vy}) - \xi_x(\xi_y-\delta_{xy}) \big) \\
    & = \sum_v p_{vy} \big( F_{x,v}(\xi) + F_{v,y}(\xi) - 2 F_{x,y}(\xi) \big) \\
    & = \sum_z p_{-z} \big( F_{x,y+z}(\xi) + F_{x+z,y}(\xi) - 2 F_{x,y}(\xi) \big)
  \end{align*}
  where we used the fact that $p$ is doubly stochastic. 
  Furthermore
  \begin{align*} 
    L_{mf} F_{x,y}(\xi) & = \sum_z p_{-z} \big( F_{x-z,y-z}(\xi) - F_{x,y}(\xi) \big), \\
    \intertext{hence}
    \big(L_{rw} + L_{mf}\big)F_{x,y}(\xi) & = \big(\widehat{L}^{(2), *}F_{\cdot,\cdot}(\xi) \big)_{x,y} .
  \end{align*}
  We also have 
  \begin{align*}
    L_{im} F_{x,y}(\xi) & = \gamma 1_{\{\xi_0=0\}} \big( \delta_{x0} \xi_y + \delta_{y0} \xi_x \big), 
    \quad \\
    L_{br} F_{x,y}(\xi) & = \gamma 1_{\{\xi_x=1\}} \delta_{xy}(1-\delta_{x0}) 
    \frac12 \Big( (\xi_x+1)\xi_x + (\xi_x-1)(\xi_x-2) - 2\xi_x(\xi_x-1) \Big) \\
    & = \gamma 1_{\{\xi_x=1\}} \delta_{xy}(1-\delta_{x0}).
  \end{align*}
  Altogether we obtain 
  \begin{align}
    \label{eq:geneqmom2}
    L F_{x,y}(\xi) = \big(\widehat{L}^{(2), *}F_{\cdot,\cdot}(\xi) \big)_{x,y} + 
    \gamma 1_{\{\xi_x=1\}} \delta_{xy}(1-\delta_{x0}) 
    + \gamma 1_{\{\xi_0=0\}} \big( \delta_{x0} \xi_y + \delta_{y0} \xi_x \big) .
  \end{align}
It then follows that that $f_{x,y}(t):= \E\big[F_{x,y}(\xi(t))]$ satisfies the equation \eqref{duhamel2}. 

\section{Auxiliary computations for Section~\ref{subsect:DoobTransform}}
\label{sect:DoobTransform-aux}

\emph{Proof that $h$ from \eqref{eq:h.eta.t} is space-time harmonic for
$(\eta_t)_{0 \le t \le T}$} :

Note that 
\begin{align}
  h(\eta^{x\to y},t) & = h(\eta,t) + p_{y,x_0}(T-t) - p_{x,x_0}(T-t), \\
  h(\eta^{+x},t) & = h(\eta,t) + p_{x,x_0}(T-t), \quad h(\eta^{-x},t)  = h(\eta,t) - p_{x,x_0}(T-t)
\end{align}
and 
\begin{align}
\frac{\partial}{\partial t} h(\eta,t) & = - \sum_x \eta_x \frac{\partial p_{x,x_0}}{\partial t}(T-t) 
= - \sum_x \eta_x \sum_z p_{xz} \big( p_{z,x_0}(T-t) - p_{x,x_0}(T-t) \big)
\end{align}
by Kolmogorov's backward equation. 
Thus 
\begin{align}
  & \sum_{x,y} \eta_x p_{xy} \big( h(\eta^{x\to y},t) - h(\eta,t) \big) 
  + \frac{\gamma}{2} \sum_x 1_{\{\eta_x=1\}} \big( h(\eta^{+x},t) + h(\eta^{-x},t) - 2 h(\eta,t) \big) 
  + \frac{\partial}{\partial t} h(\eta,t) 
  \notag \\
  & = \sum_{x,y} \eta_x p_{xy} \big( p_{y,x_0}(T-t) - p_{x,x_0}(T-t) \big) + 0  
  - \sum_x \eta_x \sum_z p_{xz} \big( p_{z,x_0}(T-t) - p_{x,x_0}(T-t) \big) = 0.
  \label{eq:hspacetimeharmonic}
\end{align}
\medskip

\noindent
\emph{Proof of the form of $\widehat{L}_t$ given in \eqref{eq:Lhat-formula}} :
\begin{align*}
  \widehat{L}_t f (\eta, t) & = \frac1{h(\eta,t)} \Big( \big(L + \tfrac{\partial}{\partial t} \big) h f \Big) (\eta,t) \\
  & = \frac1{h(\eta,t)} \bigg( 
  \sum_{x,y} \eta_x p_{xy} \big( h(\eta^{x\to y},t) f(\eta^{x\to y},t) - h(\eta,t) f(\eta,t) \big) \\
  & \hspace{5em} + \frac{\gamma}{2} \sum_x 1_{\{\eta_x=1\}} \big( h(\eta^{+x},t)f(\eta^{+x},t) + h(\eta^{-x},t)f(\eta^{-x},t) - 
  2 h(\eta,t)f(\eta,t) \big) \notag \\
  & \hspace{6em} + f(\eta,t) \frac{\partial}{\partial t} h(\eta,t) + h(\eta,t) \frac{\partial}{\partial t} f(\eta,t) 
  \bigg) \\
  & = \frac1{h(\eta,t)} \bigg( 
  \sum_{x,y} \eta_x p_{xy} \Big( \big( h(\eta,t) + p_{y,x_0}(T-t) - p_{x,x_0}(T-t) \big) f(\eta^{x\to y},t) - h(\eta,t) f(\eta,t) 
  \Big) \notag \\
  & \hspace{5em} + \frac{\gamma}{2} \sum_x 1_{\{\eta_x=1\}} \Big( \big(h(\eta,t) + p_{x,x_0}(T-t)\big) f(\eta^{+x},t) \\
  & \hspace{13em} + \big(h(\eta,t) - p_{x,x_0}(T-t)\big) f(\eta^{-x},t) - 2 h(\eta,t)f(\eta,t) \Big) \\
  & \hspace{6em} + f(\eta,t) \frac{\partial}{\partial t} h(\eta,t) \bigg) + \frac{\partial}{\partial t} f(\eta,t) \\
  & = \sum_{x,y} \eta_x p_{xy} \bigg( \Big( 1 + \frac{p_{y,x_0}(T-t) - p_{x,x_0}(T-t)}{h(\eta,t)} \Big) f(\eta^{x\to y},t) 
  - f(\eta,t) \bigg) \\
  & \hspace{2em} - f(\eta,t) \sum_x \eta_x \sum_z p_{xz} \frac{p_{z,x_0}(T-t) - p_{x,x_0}(T-t)}{h(\eta,t)} \\
  & \hspace{2em} + \frac{\gamma}{2} \sum_x 1_{\{\eta_x=1\}} \bigg( \Big( 1 + \frac{p_{x,x_0}(T-t)}{h(\eta,t)} \Big) f(\eta^{+x},t) 
  + \Big( 1 - \frac{p_{x,x_0}(T-t)}{h(\eta,t)} \Big) f(\eta^{-x},t) - 2 h(\eta,t)f(\eta,t) \bigg) \\
  & \hspace{2em} + \frac{\partial}{\partial t} f(\eta,t) \\
  & = \sum_{x,y} \eta_x p_{xy} \Big( 1 + \frac{p_{y,x_0}(T-t) - p_{x,x_0}(T-t)}{h(\eta,t)} \Big) 
  \big( f(\eta^{x\to y},t) - f(\eta,t) \big) \\
  & \hspace{2em} + \frac{\gamma}{2} \sum_x 1_{\{\eta_x=1\}} 
  \bigg( \Big( 1 + \frac{p_{x,x_0}(T-t)}{h(\eta,t)} \Big) \big( f(\eta^{+x},t) - f(\eta,t) \big) \\
  & \hspace{10em} + \Big( 1 - \frac{p_{x,x_0}(T-t)}{h(\eta,t)} \Big) \big( f(\eta^{-x},t) - f(\eta,t) \big) \bigg) 
  + \frac{\partial}{\partial t} f(\eta,t) \\
  & = \sum_{x, y\in\Z^d} \eta_x p_{xy}  \Big(1 - s_x(\eta,t) + s_x(\eta,t) \frac{p_{y,x_0}(T-t)}{p_{x,x_0}(T-t)} \Big)  
  \big( f(\eta^{x \to y},t) - f(\eta,t) \big) \notag \\ 
  & \hspace{2em} + \frac{\gamma}{2} \sum_{x} 1_{\{\eta_x=1\}} \Big( \big(1+s_x(\eta,t)\big) \big( f(\eta^{+x},t) - f(\eta,t) \big)
  \notag \\[-2ex] 
  & \hspace{10em} + \big(1-s_x(\eta,t)\big) \big( f(\eta^{-x},t) - f(\eta,t) \big) \Big) 
  + \frac{\partial}{\partial t} f(\eta,t)
\end{align*}
(recall $s_x(\eta,t)$ from \eqref{eq:Lhat-formula.sx}).
\medskip

\noindent
\emph{Proof of \eqref{eq:intLtilde=Lhat} for functions of the form \eqref{eq:projtildeexi=hateta.f-form}} : 
For $f(\xi, z, t) = f_1(\xi,t) 1_{z=z_0}$ we have 
\begin{align*} 
  g(\eta,t) = \int_{\mathbb{S} \times \Z^d} f(\xi,z,t) \, \alpha_t\big(\eta, d(\xi,z)\big) 
  =f_1(\eta,t) \frac{\eta_{z_0} p_{z_0,x_0}(T-t)}{h(\eta,t)} 
  = f_1(\eta,t) \eta_{z_0} s_{z_0}(\eta, t)
\end{align*}
and 
\begin{align} 
  & \int_{\mathbb{S} \times \Z^d} \widetilde{L}_t f(\xi,z,t) \, \alpha_t\big(\eta, d(\xi,z)\big) 
  = \sum_z \frac{\eta_{z} p_{z,x_0}(T-t)}{h(\eta,t)} \widetilde{L}_t f(\eta,z,t) \notag \\
  & = \frac{\eta_{z_0} p_{z_0,x_0}(T-t)}{h(\eta, t)} \sum_{x,y} (\eta_x-\delta_{x z_0}) p_{xy} 
  \big( f_1(\eta^{x\to y},t) - f_1(\eta,t) \big) \notag \\
  & \hspace{1em} + \sum_z \frac{\eta_{z} p_{z,x_0}(T-t)}{h(\eta, t)} \sum_y 
  p_{zy} \frac{p_{y,x_0}(T-t)}{p_{z,x_0}(T-t)} \big( f_1(\eta^{z\to y},t) 1_{y=z_0} - f_1(\eta,t) 1_{z=z_0} \big) \notag \\
  & \hspace{1em} + \sum_z \frac{\eta_{z} p_{z,x_0}(T-t)}{h(\eta, t)} 
  \bigg( \frac{\gamma}{2} \sum_{x \neq z} 1_{\{\eta_x=1\}} \big( f_1(\eta^{+x},t) 1_{z=z_0} 
  + f_1(\eta^{-x},t) 1_{z=z_0} - 2 f_1(\eta,t) 1_{z=z_0} \big) \notag \notag \\[-1ex]
  & \hspace{13em} + \gamma 1_{\eta_z=1} \big( f_1(\eta^{+z},t) 1_{z=z_0} -  f_1(\eta,t) 1_{z=z_0} \big) \bigg) \notag \\
  & \hspace{1em} + \frac{\eta_{z_0} p_{z_0,x_0}(T-t)}{h(\eta, t)} \frac{\partial}{\partial t} f_1(\eta, t) \notag \\
  & = \eta_{z_0} s_{z_0}(\eta,t) \sum_{x,y} (\eta_x-\delta_{x z_0}) p_{xy} \big( f_1(\eta^{x\to y},t) - f_1(\eta,t) \big) \notag \\ 
  & \hspace{1em} + s_{z_0}(\eta,t) \sum_z \eta_z p_{z z_0} f_1(\eta^{z \to z_0},t) 
  - \eta_{z_0} \sum_y p_{z_0 y} s_y(\eta,t) f_1(\eta,t) \notag \\
  & \hspace{1em} + \eta_{z_0} s_{z_0}(\eta,t) \bigg( 
  \gamma 1_{\{\eta_{z_0}=1\}} \big( f_1(\eta^{+z_0},t) -f_1(\eta,t) \big) \notag \notag \\
  & \hspace{8em} 
  + \sum_{x \neq {z_0}} 1_{\{\eta_x=1\}} \big( f_1(\eta^{+x},t) + f_1(\eta^{-x},t) - 2 f_1(\eta,t) \big)\bigg) \notag \\
  & \hspace{1em} + \eta_{z_0} s_{z_0}(\eta,t) \frac{\partial}{\partial t} f_1(\eta, t) \notag \\
  & = \eta_{z_0} s_{z_0}(\eta,t) \bigg( \sum_{x,y} (\eta_x - \delta_{x z_0}) p_{xy} \big( f_1(\eta^{x\to y},t) - f_1(\eta,t) \big) \notag \\
  & \hspace{8em} + \sum_z \frac{\eta_z}{\eta_{z_0}} p_{z z_0} f_1(\eta^{z \to z_0},t) 
  - f_1(\eta, t) \sum_y p_{z_0 y} \frac{p_{y, x_0}(T-t)}{p_{z_0, x_0}(T-t)} \notag \\ 
  & \hspace{8em} + \gamma 1_{\{\eta_{z_0}=1\}} \big( f_1(\eta^{+z_0},t) -f_1(\eta,t) \big) \notag \\ 
  & \hspace{8em} + \sum_{x \neq {z_0}} 1_{\{\eta_x=1\}} \big( f_1(\eta^{+x},t) + f_1(\eta^{-x},t) - 2 f_1(\eta,t) \big)
  + \frac{\partial}{\partial t} f_1(\eta, t) \bigg) .
  \label{eq:proof.intLtilde=Lhat.1}
 \end{align}
On the other side of \eqref{eq:intLtilde=Lhat} we have 
\begin{align*} 
  \widehat{L}_t g(\eta,t) & = \sum_{x,y} \eta_x \Big(1 - s_x(\eta,t) + s_x(\eta,t) \frac{p_{y,x_0}(T-t)}{p_{x,x_0}(T-t)} \Big) p_{xy} \\[-1ex]
  & \hspace{7em} \times \Big( (\eta_{z_0}+\delta_{y z_0} - \delta_{x z_0}) s_{z_0}(\eta^{x\to y},t) f_1(\eta^{x\to y}, t) 
  - \eta_{z_0} s_{z_0}(\eta,t) f_1(\eta,t) \Big) \\
  & \hspace{2em} + \frac{\gamma}{2} \sum_x1_{\{\eta_x=1\}} \bigg( (1+s_x(\eta,t)) \Big( (\eta_{z_0} + \delta_{x,z_0}) 
  s_{z_0}(\eta^{+x},t) f_1(\eta^{+x},t) \\[-2ex] 
  & \hspace{23em} - \eta_{z_0} s_{z_0}(\eta,t) f_1(\eta,t) \Big) \\
  & \hspace{10em} + (1-s_x(\eta,t)) \Big( (\eta_{z_0} - \delta_{x,z_0}) 
  s_{z_0}(\eta^{-x},t) f_1(\eta^{-x},t) \\[-2ex] 
  & \hspace{23.5em} - \eta_{z_0} s_{z_0}(\eta,t) f_1(\eta,t) \Big) \bigg) \\
  & \hspace{1em} 
  + \eta_{z_0} s_{z_0}(\eta,t) \frac{\partial}{\partial t} f_1(\eta,t) 
  + f_1(\eta,t) \eta_{z_0} \frac{\partial}{\partial t} s_{z_0}(\eta,t) \\
  & = \eta_{z_0} s_{z_0}(\eta,t) \sum_{x,y} \eta_x p_{xy} \big(1 - s_x(\eta,t) + s_y(\eta,t)\big) \\[-1ex]
  & \hspace{9em} \times 
  \Big( \frac{(\eta_{z_0}+\delta_{y z_0} - \delta_{x z_0}) s_{z_0}(\eta^{x\to y},t)}{\eta_{z_0} s_{z_0}(\eta,t)} 
  f_1(\eta^{x\to y}, t) - f_1(\eta, t) \Big) \\
  & \hspace{1em} + \eta_{z_0} s_{z_0}(\eta,t) \frac{\gamma}{2} \sum_x1_{\{\eta_x=1\}} \bigg( 
  (1+s_x(\eta,t)) \frac{\eta_{z_0} + \delta_{x,z_0}}{\eta_{z_0}} \frac{s_{z_0}(\eta^{+x},t)}{s_{z_0}(\eta,t)} f_1(\eta^{+x},t) \\
  & \hspace{14em} + (1-s_x(\eta,t)) \frac{\eta_{z_0} - \delta_{x,z_0}}{\eta_{z_0}} \frac{s_{z_0}(\eta^{-x},t)}{s_{z_0}(\eta,t)} f_1(\eta^{-x},t) \\
  & \hspace{14em} -2 f_1(\eta, t) \bigg) \\
  & \hspace{1em} 
  + \eta_{z_0} s_{z_0}(\eta,t) \Big( \frac{\partial}{\partial t} f_1(\eta,t) + f_1(\eta,t) 
  \frac{\frac{\partial}{\partial t} s_{z_0}(\eta,t)}{s_{z_0}(\eta,t)} \Big). 
\end{align*}
Note that 
\begin{align*} 
  & \big(1 - s_x(\eta,t) + s_y(\eta,t)\big) s_{z_0}(\eta^{x\to y},t) \notag \\
  & = \frac{\sum_w \eta_w p_{w x_0}(T-t) - p_{x x_0}(T-t) + p_{y x_0}(T-t)}{\sum_v \eta_v p_{v x_0}(T-t)} 
  \frac{p_{z_0 x_0}(T-t)}{\sum_w \eta_w p_{w x_0}(T-t) - p_{x x_0}(T-t) + p_{y x_0}(T-t)} \notag \\
  & = s_{z_0}(\eta,t)
\end{align*}
and 
\begin{align*} 
  & (1 \pm s_x(\eta,t)) \frac{s_{z_0}(\eta^{\pm},t)}{s_{z_0}(\eta,t)} \\
  & = \frac{p_{x,x_0}(T-t) \pm \sum_w \eta_w p_{w,x_0}(T-t)}{\sum_w \eta_w p_{w,x_0}(T-t)} 
  \frac{p_{z_0,x_0}(T-t) \cdot \sum_v \eta_v p_{v,x_0}(T-t)}{p_{z_0,x_0}(T-t) \cdot \big( p_{x,x_0}(T-t) \pm \sum_u \eta_u p_{u,x_0}(T-t)\big)} 
  = 1 .
\end{align*}
Thus 
\begin{align} 
  & \frac{\widehat{L}_t g(\eta,t)}{\eta_{z_0} s_{z_0}(\eta,t)} \notag \\
  & = \sum_{x,y} \eta_x p_{xy} \Big( \frac{\eta_{z_0} + \delta_{x,z_0}}{\eta_{z_0}} 
  \frac{\eta_{z_0}+\delta_{y z_0} - \delta_{x z_0}}{\eta_{z_0}} f_1(\eta^{x\to y}, t) 
  - \big(1 - s_x(\eta,t) + s_y(\eta,t)\big) f_1(\eta,t) \Big) \notag \\
  & \hspace{1em} + \frac{\gamma}{2} \sum_x 1_{\{\eta_x=1\}} \Big( \frac{\eta_{z_0} + \delta_{x,z_0}}{\eta_{z_0}} f_1(\eta^{+x},t) 
    + \frac{\eta_{z_0} - \delta_{x,z_0}}{\eta_{z_0}} f_1(\eta^{-x},t) -2 f_1(\eta,t) \Big) \notag \\
  & \hspace{1em} 
  + \Big( \frac{\partial}{\partial t} f_1(\eta,t) + f_1(\eta,t) 
  \frac{\frac{\partial}{\partial t} s_{z_0}(\eta,t)}{s_{z_0}(\eta,t)} \Big) \notag \\
  & = \sum_{x,y} \eta_x p_{xy} \frac{\eta_{z_0} - \delta_{x z_0}}{\eta_{z_0}} 
  \big( f_1(\eta^{x\to y}, t) - f_1(\eta,t) \big) \notag \\
  & \hspace{1em} 
  + \sum_x \eta_x p_{x z_0} \frac{1}{\eta_{z_0}} f_1(\eta^{x\to z_0}, t) 
  + \sum_{x, y} \eta_x p_{xy} \Big( s_x(\eta,t) - s_y(\eta,t) - \frac{\delta_{x z_0}}{\eta_{z_0}}\Big) f_1(\eta,t) 
  \notag \\
  & \hspace{1em} + \frac{\gamma}{2} \sum_{x \neq z_0} 1_{\{\eta_x=1\}} \big( f_1(\eta^{+x},t) + f_1(\eta^{-x},t) -2 f_1(\eta,t) \big) 
  + \gamma 1_{\{\eta_{z_0}=1\}} \big( f_1(\eta^{+z_0},t) - f_1(\eta,t) \big) 
  \notag \\
  & \hspace{1em} 
  + \Big( \frac{\partial}{\partial t} f_1(\eta,t) + f_1(\eta,t) 
  \frac{\frac{\partial}{\partial t} s_{z_0}(\eta,t)}{s_{z_0}(\eta,t)} \Big) \notag \\
  & = \sum_{x,y} (\eta_x - \delta_{x z_0}) p_{xy} \big( f_1(\eta^{x\to y}, t) - f_1(\eta,t) \big) 
  + \sum_x \eta_x p_{x z_0} \frac{1}{\eta_{z_0}} f_1(\eta^{x\to z_0}, t) \notag \\
  & \hspace{1em} + \frac{\gamma}{2} \sum_{x \neq z_0} 1_{\{\eta_x=1\}} \big( f_1(\eta^{+x},t) + f_1(\eta^{-x},t) -2 f_1(\eta,t) \big) 
  + \gamma 1_{\{\eta_{z_0}=1\}} \big( f_1(\eta^{+z_0},t) - f_1(\eta,t) \big) \notag \\
  & \hspace{1em} 
  + \frac{\partial}{\partial t} f_1(\eta,t) 
  + f_1(\eta,t) \bigg( \frac{\frac{\partial}{\partial t} s_{z_0}(\eta,t)}{s_{z_0}(\eta,t)} 
  + \sum_{x, y} \eta_x p_{xy} \Big( s_x(\eta,t) - s_y(\eta,t) - \frac{\delta_{x z_0}}{\eta_{z_0}}\Big) \bigg). 
  \label{eq:proof.intLtilde=Lhat.2}
\end{align}
We have 
\begin{align*} 
\frac{\partial}{\partial t} s_{z_0}(\eta,t) & = \frac{\partial}{\partial t} \Big( \frac{p_{z_0 x_0}(T-t)}{\sum_w \eta_w p_{w x_0}(T-t)}\Big) \\
& = \frac{-(\frac{\partial}{\partial t} p_{z_0 x_0})(T-t)}{\sum_w \eta_w p_{w x_0}(T-t)} 
+ \frac{p_{z_0 x_0}(T-t)}{\big( \sum_w \eta_w p_{w x_0}(T-t)\big)^2} 
\sum_v \eta_v \big(\frac{\partial}{\partial t} p_{v x_0}\big)(T-t),
\end{align*}
so
\begin{align*} 
  \frac{\frac{\partial}{\partial t} s_{z_0}(\eta,t)}{s_{z_0}(\eta,t)} 
  & = - \frac{(\frac{\partial}{\partial t} p_{z_0 x_0})(T-t)}{p_{z_0 x_0}(T-t)} 
  + \frac1{\sum_w \eta_w p_{w x_0}(T-t)} \sum_v \eta_v \big(\frac{\partial}{\partial t} p_{v x_0}\big)(T-t) \\
  & = - \sum_{y} p_{z_0 y} \frac{p_{y x_0}(T-t) - p_{z_0 x_0}(T-t)}{p_{z_0 x_0}(T-t)} 
  + \frac{\sum_{v,u} \eta_v p_{vu} \big(p_{u x_0}(T-t) - p_{v x_0}(T-t)\big)}{\sum_w \eta_w p_{w x_0}(T-t)} \\
  & = \sum_{y} p_{z_0 y} - \sum_{y} p_{z_0 y} \frac{p_{y x_0}(T-t)}{p_{z_0 x_0}(T-t)} 
  + \sum_{v,u} \eta_v p_{vu} \big( s_u(\eta,t) - s_v(\eta,t) \big)
\end{align*}
(where we used Kolmogorov's backward equation 
$\frac{\partial}{\partial s} p_{v,x_0}(s) = \sum_u p_{vu} \big( p_{u,x_0}(s) - p_{v,x_0}(s)\big)$)
and 
\begin{align}  
  \frac{\frac{\partial}{\partial t} s_{z_0}(\eta,t)}{s_{z_0}(\eta,t)} 
  + \sum_{x, y} \eta_x p_{xy} & \Big( s_x(\eta,t) - s_y(\eta,t) - \frac{\delta_{x z_0}}{\eta_{z_0}}\Big) 
  \notag \\ 
  \label{eq:proof.intLtilde=Lhat.3}
  & = - \sum_{y} p_{z_0 y} \frac{p_{y x_0}(T-t)}{p_{z_0 x_0}(T-t)} 
  = - \sum_y p_{z_0 y} \frac{s_y(\eta,t)}{s_{z_0}(\eta,t)}
\end{align}
Inserting \eqref{eq:proof.intLtilde=Lhat.3} into \eqref{eq:proof.intLtilde=Lhat.2} 
we obtain  
\begin{align*}
  \frac{\widehat{L}_t g(\eta,t)}{\eta_{z_0} s_{z_0}(\eta,t)} 
  & = \sum_{x,y} (\eta_x - \delta_{x z_0}) p_{xy} \big( f_1(\eta^{x\to y}, t) - f_1(\eta,t) \big) 
  + \sum_x \eta_x p_{x z_0} \frac{1}{\eta_{z_0}} f_1(\eta^{x\to z_0}, t) 
  \\ & \hspace{1em} 
  + \frac{\partial}{\partial t} f_1(\eta,t) 
  - f_1(\eta,t) \sum_{y} p_{z_0 y} \frac{p_{y x_0}(T-t)}{p_{z_0 x_0}(T-t)} 
\end{align*}
and comparing this with \eqref{eq:proof.intLtilde=Lhat.1} yields \eqref{eq:intLtilde=Lhat}. 

\end{appendix}

\bigskip


\noindent {\bf Acknowledgements.} 
R.S.\ is supported by AcRF Tier 1 grant R-146-000-220-112.  M.B.\ is
in part supported by DFG priority programme SPP 1590 Probabilistic
structures in evolution through grant BI 1058/3-2. We thank 
the Institute for Mathematical Sciences, National University of Singapore 
for hospitality and support during the program Genealogies of interacting particle 
systems, where this work was completed. 
M.B.\ would like
to thank Ted Cox, who originally posed the question concerning the
fate of the lonely branching walks, as well as Anton Wakolbinger and
Alison Etheridge for many stimulating discussions and ideas that after
many years' gestation took the form presented here.


\end{document}